\documentclass[a4paper, 12pt, reqno]{amsart}

\usepackage[T1]{fontenc}

\usepackage{lmodern}
\usepackage[hidelinks]{hyperref}
\usepackage{graphicx}
\usepackage{amsfonts}
\usepackage{amsmath, amssymb, amsthm}
\usepackage{mathrsfs}
\usepackage{mathtools}
\usepackage{longtable}
\usepackage{enumerate}
\usepackage[table]{xcolor}
\usepackage{anyfontsize}
\usepackage{tikz}
\usetikzlibrary{arrows.meta, positioning, calc}

\usepackage{upgreek}

\usepackage{url}

\usepackage{geometry}
\newcommand{\rA}{\ensuremath{\mathrm{A}}}
\newcommand{\rB}{\ensuremath{\mathrm{B}}}

\newcommand{\rd}{\ensuremath{\mathrm{d}}}

\newcommand{\ee}{\ensuremath{\mathbf{e}}}

\newcommand{\bbZ}{\ensuremath{\mathbb{Z}}}
\newcommand{\bbQ}{\ensuremath{\mathbb{Q}}}

\newcommand{\bbC}{\ensuremath{\mathbb{C}}}

\newcommand{\bbP}{\ensuremath{\mathbb{P}}}
\newcommand{\bbT}{\ensuremath{\mathbb{T}}}

\newcommand{\sfu}{\ensuremath{\mathsf{u}}}
\newcommand{\sfv}{\ensuremath{\mathsf{v}}}

\newcommand{\fu}{\ensuremath{\mathfrak{u}}}
\newcommand{\fz}{\ensuremath{\mathfrak{z}}}

\newcommand{\fB}{\ensuremath{\mathfrak{B}}}

\newcommand{\calC}{\ensuremath{\mathcal{C}}}

\newcommand{\calO}{\ensuremath{\mathcal{O}}}
\newcommand{\calP}{\ensuremath{\mathcal{P}}}

\newcommand{\Mbar}{\ensuremath{\overline{\mathcal{M}}}}

\theoremstyle{plain}
\newtheorem{proposition}{Proposition}
\newtheorem{lemma}[proposition]{Lemma}
\newtheorem{theorem}[proposition]{Theorem}
\newtheorem{corollary}[proposition]{Corollary}
\theoremstyle{definition}

\newtheorem{definition-theorem}[proposition]{Definition-Theorem}
\newtheorem{definition-proposition}[proposition]{Definition-Proposition}
\newtheorem{remark}[proposition]{Remark}

\newtheorem{introthm}{Theorem}

\theoremstyle{definition}

\theoremstyle{plain}

\numberwithin{table}{section}
\numberwithin{proposition}{section}
\numberwithin{conj}{section}    

\title[Structure of higher-genus open-closed GW theory of $X_{p}$]{Structure of higher-genus open-closed Gromov--Witten theory of $\mathcal{O}_{\mathbb{P}^{1}}(p-1)\oplus\mathcal{O}_{\mathbb{P}^{1}}(-p-1)$}

\author[S.~Guo]{Shuai Guo}
\address{School of Mathematical Sciences, Peking University, Beijing 100871, China}
\email{guoshuai@math.pku.edu.cn}

\author[J.~Xu]{Jingyi Xu}
\address{School of Mathematics and Statistics, Wuhan University, Wuhan 430072, China}
\email{jingyi.xu@whu.edu.cn}

\author[Q.~Zhang]{Qingsheng Zhang}
\address{School of Sciences, Great Bay University, Great Bay Institute for Advanced Study, Dongguan 523000, China}
\email{zhangqingsheng@gbu.edu.cn}

\begin{document}

\begin{abstract}
We study the closed and open Gromov--Witten potentials of the toric Calabi--Yau threefold
\[
X_p=\operatorname{Tot}\bigl(
\calO_{\bbP^1}(p-1)\oplus\calO_{\bbP^1}(-p-1)
\bigr),\qquad p\geq 2.
\]
We prove closed and open mirror symmetry under a nonvanishing condition on the torus weights, relating these potentials to topological recursion on the mirror curves.
We establish polynomial structures for both the higher-genus closed potentials and the stable open potentials.
We also establish double-scaling limits for topological recursion on the mirror curves.
In particular, our results for the closed potentials prove the higher-genus ansatz and the double-scaling conjecture of Caporaso--Griguolo--Mari\~no--Pasquetti--Seminara.
\end{abstract}

\maketitle

\setcounter{tocdepth}{1}
\setcounter{section}{-1}

\tableofcontents

\section{Introduction}

Computing higher-genus Gromov--Witten (GW) invariants and understanding their algebraic structure is a central problem in enumerative geometry and mirror symmetry. For Calabi--Yau (CY) threefolds, the pioneering framework of Bershadsky--Cecotti--Ooguri--Vafa (BCOV)~\cite{BCOV94} and its subsequent developments led to  a series of structural predictions: finite generation, the holomorphic anomaly equations, orbifold regularity, and the conifold gap condition~\cite{YY04,HKQ09}. Specifically, finite generation states that, upon invoking the mirror map and appropriate normalizations, higher-genus potentials can be expressed as polynomials in a finite set of generators determined entirely by genus-zero data. The holomorphic anomaly equations recursively determine the nonholomorphic dependence, leaving holomorphic ambiguities that are constrained by the orbifold regularity and the conifold gap. Together with degree bounds, these properties reduce the computation at each genus to finitely many coefficients.

In recent years, remarkable progress has been made toward proving these structural predictions. For the quintic threefold, finite generation and the holomorphic anomaly equations were established by Chang--Guo--Li~\cite{CGL21,CGL26} and Guo--Janda--Ruan~\cite{GJR18}, with the latter also proving orbifold regularity. For other geometries, Lei~\cite{Lei24} proved finite generation for three additional families of one-parameter CY hypersurfaces, while analogous polynomial structures have been uncovered for closed GW potentials and open GW differentials in various toric settings~\cite{LP18,CI21,Wang23,FRZZ19,Zhang18}. The conifold gap conjecture has likewise witnessed major breakthroughs. Guo--Janda--Ruan proposed an approach to the conifold gap via the formal quintic and verified the conjecture computationally up to genus five~\cite{GJR18}. An all-genus proof of the gap condition for both the quintic and the three other one-parameter CY hypersurfaces was recently achieved in~\cite{CGYZ26}. In parallel toric developments, Brini~\cite{Br25} established the conifold gap for the local projective plane, and Fang--Chen--Chen--Zong~\cite{FCCZ26} proved a corresponding conifold gap theorem for topological recursion on toric mirror curves near generic one-node degenerations.

In this paper, we investigate polynomial structures and critical behavior in the closed and open GW theories of the local CY threefolds
\[
X_p=\operatorname{Tot}\!\left(
\mathcal{O}_{\mathbb{P}^{1}}(p-1)\oplus
\mathcal{O}_{\mathbb{P}^{1}}(-p-1)
\right),
\qquad p\in\mathbb{Z}.
\]
Since $X_p\cong X_{-p}$, we may assume $p\geq0$. Despite its simplicity, this family is expected to exhibit analogues of \emph{all} the structural properties of the BCOV framework, but a systematic proof of these properties for $p\geq2$ has been lacking.
In these cases, there are no non-holomorphic generators. We prove that, after suitable normalization, the higher-genus closed potentials and the fixed-winding coefficients of the stable open potentials are polynomials in the global B-model coordinate $\mathbf{q}$. Boundary regularity and critical asymptotics provide counterparts of orbifold regularity and the conifold gap condition: the former yields degree bounds, while the latter determines the leading singular behavior through a universal double-scaling limit. These local curves also arise in the relation between topological strings and $q$-deformed two-dimensional Yang--Mills theory~\cite{AOSV05}.

Before explaining our main results, we note that, in the cases $p=0,1$, 
the closed and open GW potentials can be explicitly computed by the topological vertex and strip formalism~\cite{LLLZ09,IKP06} (see also~\cite{FP00, BP08} for other approaches). 
Moreover, both cases are semiprojective and thus the Remodeling Conjecture~\cite{BKMP09,BKMP10}, proved by Eynard--Orantin~\cite{EO15} and Fang--Liu--Zong~\cite{FLZ20}, applies.
Therefore, we focus on $p\geq2$ in this paper.

\subsection{Main results}
For $p\geq 2$, $X_p$ is not semiprojective and thus the remodeling theorem does not apply.
In~\cite{Mar08}, Mari\~no proposed a mirror curve and conjectured that the closed and open GW potentials of $X_p$ with zero framing can be reconstructed by topological recursion on this curve.
Eynard~\cite{Eyn08} proved Mari\~no's conjecture for the closed part.

Our first main result establishes closed and open mirror symmetry for $X_p$, with the open correspondence holding for every nonexceptional integer framing.
Here we call a complex number $f\in\bbC$ \emph{nonexceptional} if
\[
f(f+1)\bigl(1-f(p-1)\bigr)(1-fp)\neq 0.
\]

On the geometric (A-model) side, we denote by $\mathcal{F}_{g}$ and $\mathcal{F}_{g,n}^{f}$ the genus-$g$ closed and open GW potentials, respectively.
On the mirror curve (B-model) side, the potentials are denoted by $\check{\mathcal{F}}_{g}$ and $\check{\mathcal{F}}_{g,n}^{f}$ respectively.
The A-model coordinates $\mathbf{Q},X_{i}$ and B-model coordinates $\mathbf{q},\widehat{X}_{i}$ are related by the open-closed mirror map
\begin{equation}\label{equ:mirror_map}
    \mathbf{Q}=\mathbf{q}(1-\mathbf{q})^{p^{2}-1},\qquad
    \widehat{X}_{i}=-\frac{X_{i}}{(1-\mathbf{q})^{p+1}}.
\end{equation}

\begin{introthm}[Theorems~\ref{thm:FgAB} and \ref{thm:open_mirror_Xp}]
\label{thm:mirror-intro}
Under the mirror map~\eqref{equ:mirror_map}, we have
\[
\mathcal{F}_{g}(\mathbf{Q})=(-1)^{g-1}\check{\mathcal{F}}_{g}(\mathbf{q})
\]
for $g \geq 2$.
For nonexceptional $f\in\bbZ$, we have
\[
    {\mathcal{F}_{g,n}^{f}}(\mathbf{Q};X_{1},\ldots,X_{n}) = 
    (-1)^{g-1+n} \check{\mathcal{F}}_{g,n}^{f}(\mathbf{q};\widehat{X}_{1},\ldots,\widehat{X}_{n})
\]
for all $g\geq 0$ and $n\geq 1$.
For $(g,n) = (0,1)$, this equality holds for every $f\in\bbZ$.
\end{introthm}

Our second main result establishes polynomial structures for both closed and open GW potentials of $X_p$ with arbitrary framing.
For $g\geq 0$, $n\geq 1$, and $\boldsymbol{\mu}=(\mu_{1},\ldots,\mu_{n})\in\bbZ_{>0}^{n}$, we denote
\[
\mathcal{F}^{f}_{g,\boldsymbol{\mu}}(\mathbf{q}) = [X_{1}^{\mu_{1}}\cdots X_{n}^{\mu_{n}}] {\mathcal{F}_{g,n}^{f}}(\mathbf{Q};X_{1},\ldots,X_{n}),
\]
where the closed mirror map is applied.
We write $|\boldsymbol{\mu}|\coloneqq\sum_{i=1}^{n}\mu_i$.

\begin{introthm}[Theorems~\ref{thm:FgXp} and~\ref{thm:open_formal_p_polynomiality}]
\label{thm:FgXp-intro}
Let
\(
\Delta=\Delta(\mathbf{q}) = \frac{1-p^{2}\mathbf{q}}{1-\mathbf{q}}.
\)
For $g\geq 2$, we have
\[
\mathcal{F}_{g} = (-1)^{g-1} \frac{\calP_{g}(\Delta,p^{2})}{\Delta^{5g-5}},
\]
where $\calP_{g}(\Delta,p^{2})$ is a polynomial in $\Delta$ of degree at most $5g-5$, with coefficients in $(p^{2}-1)^{-(2g-2)}\bbQ[p^{2}]$.
For $n>0$, $2g-2+n>0$, and $\boldsymbol{\mu}\in\bbZ_{>0}^{n}$, we have
\[
(1-\mathbf{q})^{(p+1)|\boldsymbol{\mu}|}\mathcal{F}^{f}_{g,\boldsymbol{\mu}}(\mathbf{q}) =
\frac{\mathcal{P}_{g,\boldsymbol{\mu}}(p,f;\mathbf{q})}{(1-p^{2}\mathbf{q})^{5g-5+2n}},
\]
where
\(
\mathcal{P}_{g,\boldsymbol{\mu}}(p,f;\mathbf{q}) \in\bbQ[p,f,\mathbf{q}],
\)
and
\(
\deg_{\mathbf{q}}\mathcal{P}_{g,\boldsymbol{\mu}} \leq 5g-5+2n+|\boldsymbol{\mu}|.
\)

\end{introthm}

The closed part of Theorem~\ref{thm:FgXp-intro} proves the ansatz of~\cite[Equation~(6.2)]{CGMPS07}.

Our third main result establishes double-scaling limits of the B-model differentials and closed potentials near the critical point $\mathbf{q}=p^{-2}$.
Let $\omega_{g,n}^{f}$ denote the Eynard--Orantin differentials of the spectral curve \eqref{equ:SpecCurv}.
Fix $f\in\mathbb{C}$ and $\tau\in\mathbb{C}^{*}$, and set
\[
    \mathbf{q}_{\varepsilon} \coloneqq
    \frac{1-\varepsilon^{2}\tau}{p^{2}-\varepsilon^{2}\tau},
    \qquad \iota_{\varepsilon}(u)\coloneqq -p+\varepsilon u.
\]
Thus $\Delta(\mathbf{q}_{\varepsilon})=\varepsilon^{2}\tau$.
Up to rescaling and translation, the limiting spectral curves are those of the $(3,2)$ minimal model \eqref{equ:generic_double_scaled_curve} for $f\neq f_{\star}$ and the $(2,3)$ minimal model \eqref{equ:special_double_scaled_curve} for $f=f_{\star}$, where
\[
    f_{\star}\coloneqq \frac{2}{p-1}.
\]
Their Eynard--Orantin differentials and free energies are denoted by $\omega_{g,n}^{f,\tau,\mathrm{ds}}$ and $F_{g}^{f,\tau,\mathrm{ds}}$, respectively.
For $g\geq 2$, the free energies $F_{g}^{f,\tau,\mathrm{ds}}$ are independent of $f$ (see Lemma~\ref{lem:common_ds_free_energy}).

\begin{introthm}[Theorem~\ref{thm:double_scaling_TR}]
\label{thm:double-scaling-intro}
For $n>0$ and $2g-2+n>0$, we have
\[
    \lim_{\varepsilon\to 0} \varepsilon^{5(2g-2+n)}(\iota_{\varepsilon}^{\times n})^{*} \omega_{g,n}^{f}(\mathbf{q}_{\varepsilon}) =
    \omega_{g,n}^{f,\tau,\mathrm{ds}}.
\]
For $g\geq 2$, we have
\[
    \lim_{\varepsilon\to 0} \varepsilon^{10g-10} \check{\mathcal{F}}_{g}(\mathbf{q}_{\varepsilon}) =
    F_{g}^{f,\tau,\mathrm{ds}}.
\]
\end{introthm}

In particular, we obtain (see Corollary~\ref{cor:Pg0p})
\[
    \calP_{g}(0,p^{2}) =
    \frac{p^{6g-6}}{(p^{2}-1)^{2g-2}} \frac{1}{(3g-3)!} \left\langle\tau_{2}^{3g-3}\right\rangle_{g},
    \qquad g\geq 2.
\]
Here $\langle\tau_{2}^{3g-3}\rangle_{g}=\int_{\Mbar_{g,3g-3}}\prod_{i=1}^{3g-3}\psi_{i}^{2}$.
With the conventions in Remark~\ref{rem:notation}, this proves the conjecture of~\cite[Equation~(6.21)]{CGMPS07}; see Corollary~\ref{cor:dslimit}.
At $f=0$, we recover the double-scaling limits of the disk and annulus amplitudes in~\cite[Equations~(3.41) and~(3.43)]{Mar08}, after suitable changes of variables and normalization.

\subsection{Outline of the paper}
This paper is organized as follows.
In Section~\ref{sec:formal_open_closed}, we introduce the closed and open GW potentials of $X_{p}$ and describe their graph sum formulas.
The spectral curve and the proof of closed and open mirror symmetry are presented in Section~\ref{sec:topological_recursion_mirror}.
Section~\ref{sec:proof_main} establishes polynomiality results for the closed and open potentials.
In Section~\ref{sec:double_scaling}, we study the double-scaling limits, prove the double-scaling conjecture of~\cite{CGMPS07}, and compare the disk and annulus limits with Mari\~no's results~\cite{Mar08}.

 \subsection*{Acknowledgments}
This work was presented by the first author at Shandong University in 2025 and at Sichuan University in 2026.
The first author would like to thank the organizers for their invitations and Qile Chen, Bohan Fang, Yunfeng Jiang, Xin Wang, and Yaoxiong Wen for helpful discussions.
He also thanks Patrick Lei for taking notes on his talk~\cite{Lei26}.
The third author would like to thank Fudong Wang and Zhiyong Wang for helpful discussions.
Some of the results in this paper are contained in Jingyi Xu's PhD thesis at Peking University, which was completed at the end of 2025.

\subsection*{Use of generative AI}
The mathematical results and proofs in this paper were obtained by the authors without the use of generative AI.
Generative AI tools were used only for subsequent editorial assistance with language, notation, and references.
The authors take full responsibility for the content of this paper.

\vspace{2cm}

\section{Open-closed Gromov--Witten theory and reconstruction}
\label{sec:formal_open_closed}

We first describe the toric geometry of $X_{p}$ and define its closed theory using twisted GW theory of $\bbP^{1}$.
The associated cohomological field theory (CohFT)~\cite{KM94} is semisimple, so the Givental--Teleman reconstruction theorem ~\cite{Giv01Quant,Tel12} applies.
We then define the open theory using formal relative GW invariants and express its potentials in terms of twisted descendant invariants.
Applying the ancestor--descendant relation~\cite{KM98,Giv01Quant} and Givental--Teleman reconstruction yields a graph sum formula for the open potentials.

\subsection{Geometry and closed Gromov--Witten theory}\label{subsec:toricgeom}
We recall the description of $X_p$ in terms of its fan.
We refer the reader to~\cite{Oda88,Ful93,CLS11} for background on toric geometry, including the construction of toric varieties from fans and the orbit--cone correspondence.

We choose a basis $\{e_{1},e_{2},e_{3}\}$ of $N\cong\bbZ^{3}$ and set $M\coloneqq N^{\vee}$, $\bbT\coloneqq\operatorname{Hom}(M,\bbC^{*})$.
For $m\in M$, denote the corresponding character of $\bbT$ by $\chi^{m}$.
The fan of $X_{p}$ has ray generators
\[
    b_{1}=(-p+1,-1,1),\quad b_{2}=(0,1,1),\quad b_{3}=(1,0,1),\quad b_{4}=(0,0,1),
\]
and maximal cones $\operatorname{Cone}(b_{2},b_{3},b_{4})$ and $\operatorname{Cone}(b_{1},b_{3},b_{4})$, corresponding to the fixed points $x_{0}$ and $x_{\infty}$, respectively.
Since $\langle e_{3}^{\vee},b_{i}\rangle=1$ for all $i$, the Calabi--Yau subtorus is
\[
    \bbT'\coloneqq \ker\bigl(\chi^{e_{3}^{\vee}}\colon\bbT\longrightarrow\bbC^{*}\bigr) \cong(\bbC^{*})^{2}.
\]
The orbit closures corresponding to $\operatorname{Cone}(b_{3},b_{4})$ and $\operatorname{Cone}(b_{2},b_{4})$ are $C_{0}\cong\bbP^{1}$ and $C_{\fB}^{\circ}\cong\bbC$, respectively.
We then choose an outer Aganagic--Vafa brane $\fB$ meeting $C_{\fB}^{\circ}$.

Let $\sfu\coloneqq c_{1}^{\bbT'}(T_{x_{0}}C_{\fB}^{\circ})$ and $\sfv\coloneqq c_{1}^{\bbT'}(T_{x_{0}}C_{0})$. Then $H^{*}_{\bbT'}(\mathrm{pt})=\bbC[\sfu,\sfv]$.
The $\bbT'$-weights on $T_{x_{0}}X_{p}$ are $\sfv$, $\sfu$, and $-\sfu-\sfv$, and those on $T_{x_{\infty}}X_{p}$ are $-\sfv$, $\sfu-(p-1)\sfv$, and $-\sfu+p\sfv$, with the weights along $C_{0}$ listed first.
An integer framing $f$ determines the subtorus
\[
    \bbT_{f}\coloneqq \ker\bigl(\chi^{e_{2}^{\vee}-fe_{1}^{\vee}}\big|_{\bbT'}\bigr) \subset\bbT',
\]
and restriction to $\bbT_{f}$ gives $\sfv=f\sfu$ and $H^{*}_{\bbT_{f}}(\mathrm{pt})=\bbC[\sfu]$.

\subsubsection{Twisted Gromov--Witten theory}
\label{subsec:closed_GW}

Let $E_{p}\coloneqq\calO_{\bbP^{1}}(p-1)\oplus\calO_{\bbP^{1}}(-p-1)$.
Denote by $t$ the K\"ahler parameter of $C_{0}$, and set
\[
    \mathbf{Q}\coloneqq(-1)^{p+1}\ee^{-t}.
\]

The closed GW theory of $X_{p}$ is defined as $\bbT'$-equivariant GW theory of $\bbP^{1}$ twisted by $e_{\mathbb{T}'}^{-1}(E_{p})$.
We choose $H\in H^{2}_{\mathbb{T}'}(\mathbb{P}^{1})$ such that
\[
    H|_{x_{0}}=\frac{\sfv}{2},\qquad H|_{x_{\infty}}=-\frac{\sfv}{2}.
\]
The state space is then
\[
    V_{\bbT'}\coloneqq H^{*}_{\bbT'}(\bbP^{1})\otimes_{\bbC[\sfu,\sfv]}\bbC(\sfu,\sfv) = \frac{\bbC(\sfu,\sfv)[H]}{(H^{2}-\sfv^{2}/4)}.
\]
The universal curve and universal map are denoted by
\[
    \pi\colon\mathcal{C}_{g,n,d}\longrightarrow \Mbar_{g,n}(\bbP^{1},d),
    \qquad F\colon\mathcal{C}_{g,n,d}\longrightarrow\bbP^{1},
\]
and the forgetful and evaluation maps by
\[
    \operatorname{st}\colon \Mbar_{g,n}(\bbP^{1},d)\longrightarrow\Mbar_{g,n},
    \qquad \operatorname{ev}_{i}\colon \Mbar_{g,n}(\bbP^{1},d)\longrightarrow\bbP^{1}.
\]
The associated $\bbT'$-equivariant CohFT is
\[
    \Omega^{\bbT'}_{g,n}(\alpha_{1},\ldots,\alpha_{n})
    \coloneqq
    \sum_{d\geq 0}\ee^{-dt}\operatorname{st}_{*}
    \Bigl(
    e_{\bbT'}^{-1}\bigl(\mathrm{R}^{\!\bullet}\pi_{*}F^{*}E_{p}\bigr)
    \prod_{i=1}^{n}\operatorname{ev}_{i}^{*}\alpha_{i}
    \frown[\Mbar_{g,n}(\bbP^{1},d)]^{\mathrm{vir}}_{\bbT'}
    \Bigr),
\]
with pairing
\[
    \langle a,b\rangle_{\bbT'} \coloneqq\int_{\bbP^{1}} a\smile b\smile e_{\bbT'}^{-1}(E_{p}).
\]

We now regard \(f\) as a formal parameter and define
\[
    \Omega^{f} \coloneqq \left.\left(\left.\Omega^{\bbT'}\right|_{\sfv=f\sfu}\right)\right|_{\sfu=1}.
\]
It is a CohFT with state space $V_{f}\coloneqq{\bbC(f)[H]}/(H^{2}-f^{2}/4)$. 
Localization then gives the pairing on $V_{f}$ as
\[
    \langle a,b\rangle_{f} \coloneqq -\frac{a|_{H=f/2}\,b|_{H=f/2}}{f(f+1)}+\frac{a|_{H=-f/2}\,b|_{H=-f/2}}{f(1-f(p-1))(1-fp)}.
\]
Note that this pairing is well-defined and nondegenerate at a nonexceptional value of $f$.

For $d\geq 1$ or $g\geq 2$, we define the GW invariants of $X_{p}$ by
\[
    N_{g,d} \coloneqq \int_{[\Mbar_{g,0}(\bbP^{1},d)]^{\mathrm{vir}}_{\bbT'}}e_{\bbT'}^{-1}\left(\mathrm{R}^{\!\bullet}\pi_{*}F^{*}E_{p}\right)\in\bbQ(\sfu,\sfv).
\]
For $d\geq 1$, all $\bbT'$-fixed maps are supported on $C_{0}$, so localization identifies $N_{g,d}$ with the formal relative invariant of the formal toric Calabi--Yau graph of $X_{p}$ with empty partitions on all external legs.
Applying~\cite[Theorem~4.8]{LLLZ09} shows that $N_{g,d}\in\bbQ$ and is independent of $\sfu,\sfv$.
For $d=0$, virtual localization on $\Mbar_{g}\times\bbP^{1}$ implies (cf. \cite{FP00})
\[
    N_{g,0} = \frac{(-1)^{g}}{(2g-2)!}\frac{|B_{2g-2}|}{2g-2}\frac{|B_{2g}|}{2g},\qquad g\geq 2,
\]
where $B_{2k}$ are Bernoulli numbers.

To specify the potentials in genus zero and genus one, we first compute the degree-zero part of the Yukawa coupling, we have
\[
    C_f \coloneqq \langle H^{2},H\rangle_{f}=-\frac{f^{2}}{8}\left(\frac{1}{f+1}+\frac{1}{(1-f(p-1))(1-fp)}\right).
\]
In genus one, we have
\[
    \left.\int_{\Mbar_{1,1}}\Omega^{f}_{1,1}(H)\right|_{\mathbf{Q}=0}
    = -\frac{1}{24}+\frac{1}{6}C_f.
\]
Therefore, we define the genus-$g$ GW potential by
\[ \mathcal{F}_{g}(t) \coloneqq
\begin{cases}
    \frac{(-t)^{3}}{6}C_f
    +\sum_{d\geq 1}N_{g,d}\ee^{-dt}, & g = 0, \\[0.6em]
    (-t)\left(-\frac{1}{24}+\frac{1}{6}C_f\right)
    +\sum_{d\geq 1}N_{g,d}\ee^{-dt}, & g = 1, \\[0.6em]
    \sum_{d\geq 0} N_{g,d} \ee^{-dt}, & g \geq 2.
\end{cases}
\]
We note here that for $g=0,1$, the degree-zero terms of $\mathcal{F}_{g}$ depend on $f$ through $C_f$.
With this definition,
\[
    \int_{\Mbar_{0,3}}\Omega_{0,3}^{f}(H,H,H)=-\partial_{t}^{3}\mathcal{F}_{0}(t),\quad \int_{\Mbar_{1,1}}\Omega_{1,1}^{f}(H)=-\partial_{t}\mathcal{F}_{1}(t),
\]
while for $g\geq 2$,
\[
    \mathcal{F}_{g}(t)=\int_{\Mbar_{g}}\Omega_{g,0}^{f}.
\]

\subsubsection{Quantum product and semisimplicity}

For $d\geq 1$, the genus-zero invariants are
\[
    N_{0,d} = \frac{(-1)^{(p+1)d-1}}{d^{3}}\binom{p^{2}d-1}{d-1};
\]
see~\cite[Equation~(4.53)]{CGMPS07} or~\cite[Theorem~3.2]{GNS23}. Using the closed mirror map in \eqref{equ:mirror_map}, we obtain the Yukawa coupling 
\[
    -\partial_{t}^{3}\mathcal{F}_{0}(t)=C_f+\frac{\mathbf{q}}{p^{2}\mathbf{q}-1}.
\]
This determines the quantum product $\ast_{f}$ for $\Omega^{f}$ in the basis $\{\mathbf{1},H\}$ as
\begin{equation}\label{equ:quantum_product_H}
    \renewcommand{\arraystretch}{1.5}
    \setlength{\arraycolsep}{6pt}
    H\,\ast_{f}=\left(
    \begin{array}{c@{\;\;}c}
    0 & \dfrac{f^{2}(1-\mathbf{q})+(2-(p-1)f)^{2}\mathbf{q}}
    {4(1-p^{2}\mathbf{q})} \\
    1 & \dfrac{p(2-(p-1)f)\mathbf{q}}{1-p^{2}\mathbf{q}}
    \end{array}
    \right).
\end{equation}
We denote the two eigenvalues of \eqref{equ:quantum_product_H} by $L_{0}^{f}$ and $L_{1}^{f}$, labeled so that
\[
    L_{0}^{f}\big|_{\mathbf{q}=0}=\frac{f}{2}\quad\text{at }x_{0},\qquad
    L_{1}^{f}\big|_{\mathbf{q}=0}=-\frac{f}{2}\quad\text{at }x_{\infty}.
\]
Thus the CohFT $\Omega^{f}$ is semisimple.
The corresponding canonical basis is
\[
    e_{\alpha}^{f}=\frac{H-L_{1-\alpha}^{f}\mathbf{1}}{L_{\alpha}^{f}-L_{1-\alpha}^{f}},\qquad \alpha=0,1.
\]
We denote its dual basis by $\{(e^{f})^{\alpha}\}_{\alpha=0,1}$ and set $\Delta^{f}_{\alpha} \coloneqq \langle e^{f}_{\alpha},e^{f}_{\alpha}\rangle_{f}^{-1}$.
Then
\begin{equation}\label{equ:Delta_f}
\Delta_{\alpha}^{f} = -p(2-(p-1)f)
\left((L_{\alpha}^{f})^{2}+\frac{f^{2}}{4}\right)
-\frac{(2-(p-1)f)^{2}+(p^{2}-1)f^{2}}{2}L_{\alpha}^{f}.
\end{equation}
At $\mathbf{q}=0$, we have
\[
    \left.(\Delta^{f}_{0})^{-1}\right|_{\mathbf{q}=0} = -\frac{1}{f(f+1)},\qquad \left.(\Delta^{f}_{1})^{-1}\right|_{\mathbf{q}=0} = \frac{1}{f(1-(p-1)f)(1-pf)}.
\]
The normalized canonical basis is defined as $\bar{e}^{f}_{\alpha}\coloneqq(\Delta^{f}_{\alpha})^{1/2}e^{f}_{\alpha}$.

\subsection{Givental--Teleman reconstruction}\label{subsec:GT_reconstruction}

Let $\Omega$ be a semisimple CohFT with canonical idempotents $\{e_\beta\}_{\beta=1}^{N}$.
Set $\Delta_\beta\coloneqq\langle e_\beta,e_\beta\rangle^{-1}$ and $\bar{e}_\beta\coloneqq\Delta_\beta^{1/2}e_\beta$, where $\langle\, ,\,\rangle$ is the pairing of $\Omega$.
The Givental--Teleman reconstruction theorem expresses $\Omega$ as a sum over stable graphs~\cite{Giv01Quant,Tel12,PPZ15}. 

Let $\Omega^{\mathrm{pt},N}$ be the direct sum of $N$ one-dimensional trivial CohFTs, written in the normalized canonical basis $\{\bar{e}_{\beta}\}_{\beta=1}^{N}$.
The units of $\Omega^{\mathrm{pt},N}$ and $\Omega$ are, respectively,
\[
\bar{\mathbf{1}}\coloneqq\sum_{\beta=1}^{N}\bar{e}_{\beta},\qquad\mathbf{1}=\sum_{\beta=1}^{N}\Delta_{\beta}^{-1/2}\bar{e}_{\beta}.
\]
The reconstruction then takes the form
\begin{equation}
\Omega=R\cdot T_{R}\cdot\Omega^{\mathrm{pt},N},\qquad T_{R}(z)\coloneqq z\bigl(\bar{\mathbf{1}}-R^{-1}(z)\mathbf{1}\bigr).
\end{equation}

In the graph sum, an ordinary leg labeled by $\alpha_{i}$ carries $R^{-1}(\psi_{i})\alpha_{i}$, while an edge $e$ joining $v_{1}$ and $v_{2}$ carries
\[
\sum_{\beta}\frac{e_{\beta}\otimes e^{\beta}-R^{-1}(\psi_{(e,v_{1})})e_{\beta}\otimes R^{-1}(\psi_{(e,v_{2})})e^{\beta}}{\psi_{(e,v_{1})}+\psi_{(e,v_{2})}}.
\]
Here $\{e_{\beta}\}$ and $\{e^{\beta}\}$ are dual bases. The translation inserts $T_{R}(\psi)$ at $k$ additional marked points, followed by pushforward along the forgetful map and summation over $k\geq 0$ with weight $1/k!$.

Let $\mathsf{G}_{g,n}$ be the set of stable graphs of genus $g$ with $n$ labeled legs.
For $\Gamma\in\mathsf{G}_{g,n}$ and $A_i(z)\in V_f[\![z]\!]$, we denote by
\[
\operatorname{Cont}^{f}_{\Gamma}\bigl[A_1(\psi_1),\ldots,A_n(\psi_n)\bigr]
\]
the integral of the contribution of $\Gamma$ in the Givental--Teleman reconstruction of $\Omega^f$, with leg factors $A_i(\psi_i)$, omitting the automorphism factor $1/|\operatorname{Aut}(\Gamma)|$.

We determine the $R$-matrix $R^{f}$ that reconstructs $\Omega^{f}$ by solving the quantum differential equation (QDE), with the initial condition at $\mathbf{q}=0$ given by quantum Riemann--Roch.
Let
\[
    D\coloneqq -\partial_{t} = \mathbf{Q}\frac{\partial}{\partial\mathbf{Q}} = \frac{1-\mathbf{q}}{1-p^{2}\mathbf{q}}\, \mathbf{q}\frac{\partial}{\partial\mathbf{q}}.
\]
For $i,\alpha \in \{0,1\}$, we write
\[
R^{f}_{i\bar{\alpha}}(z) \coloneqq \left\langle R^{f}(z)^{*} H^{i},\bar{e}_{\alpha}^{f}\right\rangle_{f},\quad (R^{f})_{i}^{\ \alpha}(z) \coloneqq \left\langle R^{f}(z)^{*} H^{i},(e^{f})^{\alpha}\right\rangle_{f}.
\]
In these components, the quantum differential equation takes the form
\[
(zD+L_{\alpha}^{f})R^{f}_{0\bar{\alpha}}=R^{f}_{1\bar{\alpha}}
\]
and
\begin{equation}\label{equ:QDE_R0b}
\left[
\left(zD+L^{f}_{\alpha}\right)^{2}
-\frac{p(2-(p-1)f)\mathbf{q}}{1-p^{2}\mathbf{q}}
\left(zD+L^{f}_{\alpha}\right)
-\frac{f^{2}(1-\mathbf{q})+(2-(p-1)f)^{2}\mathbf{q}}
{4(1-p^{2}\mathbf{q})}
\right]
R^{f}_{0\bar{\alpha}}=0.
\end{equation}
Equivariant localization and quantum Riemann--Roch (cf. \cite{CG07}) give the initial condition
\begin{equation}
    \left.(R^{f})_{0}^{\ \alpha}(z)\right|_{\mathbf{q}=0} = \mathcal{R}^{f}_{\alpha}(z) \coloneqq \exp\left(-\sum_{k>0}\frac{B_{2k}N^{f}_{2k-1,\alpha}}{2k(2k-1)}z^{2k-1}\right),\qquad \alpha=0,1,
\end{equation}
where
\[
\begin{aligned}
    N^{f}_{m,0}&\coloneqq f^{-m}+1+(-1-f)^{-m},\\
    N^{f}_{m,1}&\coloneqq (-f)^{-m}+(1-(p-1)f)^{-m}+(pf-1)^{-m}.
\end{aligned}
\]
These equations and initial values uniquely determine the $R$-matrix $R^{f}(z)$.

\subsection{Open Gromov--Witten invariants and topological vertex}
\label{subsec:open_GW}

Fix the outer brane $\fB$ meeting $C_{\fB}^{\circ}$ and an integral framing $f$.
We define the open GW invariants of $(X_p,\fB,f)$ using formal relative GW theory, following~\cite{LLLZ09,FL13}.
We also recall the topological vertex gluing formula, which provides an effective way to compute these invariants and will be used later to study their dependence on $p$ and $f$.

Following~\cite{LLLZ09,FL13}, we associate to the framed brane $(\fB,f)$ a relative formal toric Calabi--Yau threefold $(\widehat{Y}_{p,f},\widehat{D}_{p,f})$. In the toric graph of $X_p$, we replace the noncompact edge corresponding to $C_{\fB}^{\circ}$ by a compact edge $C_{\fB}\cong\bbP^1$ ending at a univalent vertex. The relative divisor $\widehat{D}_{p,f}$ corresponds to this new vertex, and the other noncompact edges are left unchanged.
For $d\geq 0$ and a nonempty partition $\mu=(\mu_{1}\geq\cdots\geq\mu_{n}>0)$, we consider the effective curve class
\[
    \widetilde{\beta}(d,\mu)\coloneqq d[C_{0}]+|\mu|[C_{\fB}],
\]
where $|\mu|\coloneqq\sum_{i=1}^{n}\mu_{i}$.
The moduli space of genus-$g$ relative stable morphisms with effective class $(\widetilde{\beta}(d,\mu),\mu)$ is denoted by $\Mbar_{g,\widetilde{\beta}(d,\mu),\mu}(\widehat{Y}_{p,f},\widehat{D}_{p,f})$.
It carries a $\bbT'$-equivariant perfect obstruction theory of virtual dimension $n$.
Let $\widehat{L}\subset\widehat{D}_{p,f}$ be the $\bbT'$-invariant divisor specified by the framing. We define the formal relative invariant by
\[
\begin{aligned}
N^{\mathrm{rel}}_{g,d,\mu}
\coloneqq{}&
\frac{1}{|\operatorname{Aut}(\mu)|}
\int_{
[\Mbar_{g,\widetilde{\beta}(d,\mu),\mu}
(\widehat{Y}_{p,f},\widehat{D}_{p,f})^{\bbT'}]^{\mathrm{vir}}}
\frac{
\prod_{i=1}^{\ell(\mu)}\operatorname{ev}_{i}^{*}
c_{1}^{\bbT'}\!\left(
\calO_{\widehat{D}_{p,f}}(\widehat{L})
\right)}
{e_{\bbT'}(N^{\mathrm{vir}})}.
\end{aligned}
\]
Here $\ell(\mu)$ is the length of partition $\mu$, $\operatorname{ev}_{i}$ denotes the evaluation at the $i$-th marked point, and $N^{\mathrm{vir}}$ is the virtual normal bundle of the fixed locus.
According to~\cite[Theorem~4.8]{LLLZ09}, $N^{\mathrm{rel}}_{g,d,\mu}$ is independent of the equivariant parameters.
The associated open GW invariant is defined by
\begin{equation}\label{equ:open_GW_invariant}
    N_{g,d,\mu}^{f} \coloneqq (-1)^{(f+1)|\mu|}N^{\mathrm{rel}}_{g,d,\mu}.
\end{equation}
In particular, by~\cite[Corollary~3.6]{FL13}, we have
\begin{equation}\label{equ:degree_zero_open_disk}
N_{0,0,(w)}^{f} = -\frac{1}{w^{2}}\binom{(f+1)w-1}{w-1},\qquad w\geq 1.
\end{equation}
For $g\geq 0$ and $n\geq 1$, we define the open GW potential by
\[
{\mathcal{F}^{f}_{g,n}}
(\mathbf{Q};X_{1},\ldots,X_{n})
\coloneqq
\sum_{d\geq 0}\sum_{\mu_{1},\ldots,\mu_{n}\geq 1}
|\operatorname{Aut}(\mu)|
(-1)^{(p+1)d}
N_{g,d,\mu}^{f}\,
\mathbf{Q}^{d}\prod_{i=1}^{n}X_{i}^{\mu_{i}}.
\]
Here $\mu$ is the partition with parts $\mu_1,\ldots,\mu_n$.

We next introduce the open-closed partition function.
The partition function of the positive-degree closed invariants is\footnote{Our parameter $\hbar$ is related to
the parameter $\lambda$ in~\cite{LLLZ09} by
$\hbar=\sqrt{-1}\lambda$.}
\[
Z_{\mathrm{cl},>0}\coloneqq
\exp\!\left(
\sum_{g\geq 0}\sum_{d\geq 1}
(-1)^{g-1+(p+1)d}N_{g,d}\,
\hbar^{2g-2}\mathbf{Q}^{d}
\right).
\]

For the open sector, we use power-sum variables $\mathbf{p}=(p_{1},p_{2},\ldots)$ and write $p_{\mu}\coloneqq\prod_{i=1}^{\ell(\mu)}p_{\mu_{i}}$.
The open-closed partition function is then defined by
\[
Z^{f}(\hbar,\mathbf{Q};\mathbf{p})
\coloneqq
Z_{\mathrm{cl},>0}
\exp\!\left(
\sum_{\substack{g,d\geq 0\\ \mu\neq\varnothing}}
(-1)^{g-1+(p+1)d}
N_{g,d,\mu}^{f}\,\mathbf{Q}^{d}
\hbar^{2g-2+\ell(\mu)}p_{\mu}
\right).
\]
The corresponding connected generating function is
\begin{equation}\label{equ:connected_generating_function}
F^{f}(\hbar,\mathbf{Q};\mathbf{p})
\coloneqq\log Z^{f}(\hbar,\mathbf{Q};\mathbf{p}).
\end{equation}
The open potentials are recovered from the connected generating function by
\begin{equation}\label{equ:open_potential_recovery}
\begin{aligned}
\mathcal{F}^{f}_{g,n}(\mathbf{Q};X_{1},\ldots,X_{n})
={}&(-1)^{g-1}\left[\hbar^{2g-2+n}\right]
\left.
\bigg\{
\prod_{i=1}^{n}
\Big(
\sum_{k\geq 1}X_i^k\frac{\partial}{\partial p_k}
\Big)
\bigg\}
F^{f}(\hbar,\mathbf{Q};\mathbf{p})
\right|_{\mathbf{p}=0}.
\end{aligned}
\end{equation}

The topological vertex gluing formula is expressed in terms of Schur functions. The change of basis is
\[
    p_{\mu}=\sum_{\rho\vdash|\mu|}\chi_{\rho}(\mu)s_{\rho}(\mathbf{p}),\qquad s_{\rho}(\mathbf{p})=\sum_{\mu\vdash|\rho|}\frac{\chi_{\rho}(\mu)}{z_{\mu}}p_{\mu}.
\]
Here $z_{\mu}\coloneqq\left(\prod_{i=1}^{\ell(\mu)}\mu_{i}\right)|\operatorname{Aut}(\mu)|$, and $\chi_{\rho}(\mu)$ denotes the character of the irreducible representation of $S_{|\mu|}$ indexed by $\rho$, evaluated on the conjugacy class determined by $\mu$.
We set $\mathsf{q}\coloneqq\ee^{\hbar}$ and
$\boldsymbol{\varrho}\coloneqq(-\tfrac12,-\tfrac32,\ldots)$.
The vertex contributions are expressed in terms of the functions $W_{\rho}$ and $W_{\rho,\nu}$, given by
\begin{equation}\label{equ:one_two_partition_W}
W_{\rho}(\mathsf{q})
\coloneqq s_{\rho}(\mathsf{q}^{\boldsymbol{\varrho}}),
\qquad
W_{\rho,\nu}(\mathsf{q})
\coloneqq
W_{\rho}(\mathsf{q})\,
s_{\nu}(\mathsf{q}^{\rho+\boldsymbol{\varrho}}),
\end{equation}
where $\mathsf{q}^{\rho+\boldsymbol{\varrho}} = (\mathsf{q}^{\rho_{1}-\frac{1}{2}},\mathsf{q}^{\rho_{2}-\frac{3}{2}},\ldots)$.

We label the compact edge of the formal toric Calabi--Yau graph of $X_p$ by $\nu$ and all external legs by the empty partition.
The gluing formula \cite[Theorem~7.5 and Corollary~7.6]{LLLZ09}, together with the vertex evaluation \cite[Corollary~8.8]{LLLZ09}, gives
\begin{equation}\label{equ:closed_sector_of_open_vertex}
Z_{\mathrm{cl},>0} =
\sum_{\nu}
\mathbf{Q}^{|\nu|}
\mathsf{q}^{p\kappa_{\nu}/2}
W_{\nu}(\mathsf{q})W_{\nu^{t}}(\mathsf{q}),
\end{equation}
where $\kappa_{\nu}\coloneqq\sum_{i}\nu_{i}(\nu_{i}-2i+1)$ and $\nu^{t}$ denotes the transpose of $\nu$.
Labeling the external leg corresponding to $C_{\mathfrak B}^{\circ}$ by $\rho$ instead gives
\begin{equation}\label{equ:open_vertex_gluing}
Z^{f}
=
\sum_{\rho,\nu}
\mathbf{Q}^{|\nu|}
\mathsf{q}^{(f\kappa_{\rho}+p\kappa_{\nu})/2}
W_{\rho,\nu}(\mathsf{q})
W_{\nu^{t}}(\mathsf{q})
s_{\rho}(\mathbf{p}).
\end{equation}

\subsection{Graph sums for open potentials}
\label{subsec:open_leaves}

By~\eqref{equ:degree_zero_open_disk}, the degree-zero disk potential is
\[
\sum_{w\geq 1}N_{0,0,(w)}^{f}X^{w}.
\]
{
Using the polynomial expression in~\eqref{equ:degree_zero_open_disk}, we write the $\bbT'$-equivariant disk factor as
}
\begin{equation}\label{equ:full_torus_disk_specialization}
-\frac{1}{(w-1)!\,\sfu^{w+1}}
\prod_{a=1}^{w-1}(w\sfv+a\sfu)
=\left.\frac{w^{2}}{\sfu^{2}}N_{0,0,(w)}^{f}\right|_{f=\sfv/\sfu}.
\end{equation}

For $g\geq 0$ and $n>0$, we define the series
\begin{equation}\label{equ:full_torus_open_localization}
\begin{aligned}
\mathcal{F}^{\bbT'}_{g,n}
(\mathbf{Q};X_{1},\ldots,X_{n};\sfu,\sfv)
\coloneqq{}&
\sum_{\substack{d\geq 0\\ d>0\ \mathrm{or}\ 2g-2+n>0}}
\sum_{\mu_{1},\ldots,\mu_{n}\geq 1}
(-1)^{(p+1)d}\mathbf{Q}^{d}\\
&\quad\cdot
\int_{[\Mbar_{g,n}(\bbP^{1},d)]^{\mathrm{vir}}_{\bbT'}}
e_{\bbT'}^{-1}\left(\mathrm{R}^{\!\bullet}\pi_{*}F^{*}E_{p}\right)
\prod_{i=1}^{n}
\frac{\operatorname{ev}_{i}^{*}\Phi_{0}}
{1-(\mu_{i}/\sfu)\widehat{\psi}_{i}}\\
&\quad\cdot
\prod_{i=1}^{n}\left(
\frac{\mu_{i}^{2}X_{i}^{\mu_{i}}}{\sfu^{2}}
\left.N_{0,0,(\mu_i)}^{f}\right|_{f=\sfv/\sfu}
\right).
\end{aligned}
\end{equation}
Here $\Phi_0$ is the product of $e_{\mathbb T'}(E_p)$ and the $\mathbb T'$-equivariant Poincaré dual of $x_0$ with respect to the untwisted pairing.
The class $\widehat\psi_i$ is the cotangent line class at the $i$-th marked point on the moduli space of stable maps.

We specialize to $\mathsf v=f\mathsf u$ and then $\mathsf u=1$. For $(g,n)=(0,1),(0,2)$, the specialized series in the localization identity below is understood to include the degree-zero contributions prescribed by the unstable conventions of~\cite[Section~6.1]{LLLZ09}.

We apply the localization calculation in the proof of~\cite[Proposition~3.1]{FL13} to the relative formal toric Calabi--Yau threefold associated with the chosen framed outer brane in $X_p$.
The relevant fixed loci have projective coarse moduli spaces~\cite[Section~4.4]{LLLZ09}, so semiprojectivity is not required for this calculation.
Since the $\bbT'$-fixed stable maps to $X_p$ have image contained in the zero section, the corresponding descendant integrals agree with the twisted integrals in~\eqref{equ:full_torus_open_localization}.
The same argument as in the derivation of~\cite[Corollary~3.3]{FL13} shows that the specialization $\sfv=f\sfu$ is well-defined coefficientwise for every integral framing $f$.
The specialization is taken coefficientwise after summing all localization contributions. With these conventions, the localization calculation gives
\begin{equation}\label{equ:formal_relative_localization}
\left.
\left(
\left.\mathcal{F}^{\bbT'}_{g,n}
(\mathbf{Q};X_1,\ldots,X_n;\sfu,\sfv)
\right|_{\sfv=f\sfu}
\right)
\right|_{\sfu=1}
=
\mathcal{F}^{f}_{g,n}(\mathbf{Q};X_1,\ldots,X_n).
\end{equation}

To obtain the graph sum, we apply the ancestor--descendant relation using the $S$-operator of the $\bbT'$-equivariant CohFT $\Omega^{\bbT'}$ (cf.\ \cite[Appendix~2]{CG07}).
For $a,b\in V_{\bbT'}$, it is defined by
\[
\bigl\langle a,S^{\bbT'}(z)b\bigr\rangle_{\bbT'}
=\langle a,b\rangle_{\bbT'}
+\sum_{d\geq 1}\ee^{-dt}
\left\langle a,\frac{b}{z-\widehat{\psi}}\right\rangle^{\mathrm{tw}}_{0,2,d},
\]
where the brackets on the right denote descendant correlators of the $\bbT'$-equivariant twisted theory defined above.

For a nonexceptional integral framing $f$, we denote by $\Phi_{0}^{f}$ and $S^{f}(z)$ the specializations of $\Phi_{0}$ and $S^{\bbT'}(z)$ at $\sfv=f\sfu$ and $\sfu=1$.
Applying the ancestor--descendant relation and summing over winding
numbers gives the ancestor insertion
\begin{equation}\label{equ:A_model_open_input_series}
\mathcal{U}_{\rA}^{f}(\mathbf{Q};X,z) \coloneqq
\left(
    \sum_{w\geq 1}
    \frac{w^{2}N_{0,0,(w)}^{f}X^{w}}{1-wz}
    S^{f}(z)\Phi_{0}^{f}
\right)_{+},
\end{equation}
where $(\cdot)_{+}$ denotes the part with nonnegative powers of $z$.
Equivalently,
\begin{equation}\label{equ:A_model_open_input_expansion}
\mathcal{U}_{\rA}^{f}(\mathbf{Q};X,z)
=\sum_{w\geq 1}\frac{V_{w}^{f}X^{w}}{1-wz},
\qquad
V_{w}^{f}\coloneqq w^{2}N_{0,0,(w)}^{f}S^{f}(1/w)\Phi_{0}^{f}.
\end{equation}
The A-model leg contribution is
\[
    \mathcal{L}_{\rA}^{f}(\mathbf{Q};X,z) \coloneqq R^{f}(z)^{-1} \mathcal{U}_{\rA}^{f}(\mathbf{Q};X,z).
\]

\begin{proposition}
\label{prop:formal_relative_open_graph_sum}
Let $n>0$, $2g-2+n>0$, and let $f$ be a nonexceptional integral framing.
Then
\begin{equation}
    {\mathcal{F}^{f}_{g,n}}(\mathbf{Q};X_{1},\ldots,X_{n}) =
\sum_{\Gamma\in\mathsf{G}_{g,n}}
\frac{1}{|\operatorname{Aut}(\Gamma)|}
\operatorname{Cont}^{f}_{\Gamma}
\left[
\left(
\mathcal{L}_{\rA}^{f}
(\mathbf{Q};X_{i},\psi_{i})
\right)_{i=1}^{n}
\right].
\end{equation}
\end{proposition}

\begin{proof}
{By~\eqref{equ:formal_relative_localization}}, the open potential ${\mathcal{F}^{f}_{g,n}}$ is the specialization of the descendant series \eqref{equ:full_torus_open_localization}.
Applying the ancestor--descendant relation and Givental--Teleman reconstruction to the CohFT $\Omega^{\bbT'}$ expresses this series as a sum over stable graphs.
After specialization to $\sfv=f\sfu$ and $\sfu=1$, the leg contributions become $R^{f}(z)^{-1}\mathcal{U}_{\rA}^{f}=\mathcal{L}_{\rA}^{f}$, giving the stated formula.
\end{proof}
\section{Topological recursion and open-closed mirror symmetry}
\label{sec:topological_recursion_mirror}

In this section, we first recall the relation between topological recursion of spectral curves and semisimple CohFTs.
We then prove closed and open mirror symmetry for $X_{p}$ by comparing the A-model and B-model graph sums.

\subsection{Topological recursion and cohomological field theories}
\label{subsec:TR_CohFT}
The topological recursion, introduced by Eynard--Orantin \cite{EO07} in the context of matrix models, produces multidifferentials from spectral curve data. 
By definition, the spectral curve data is given by
    \(
    \mathcal{C}\coloneqq(\Sigma,x,y), 
    \)
    where $\Sigma$ is a Riemann surface, $x,y$ are two functions on $\Sigma$ such that $\rd x$ is meromorphic and the zeros $\{z^{\beta}\}_{\beta=1}^{N}$ of $\rd x$ are simple, and $\rd y$ is holomorphic and nonzero near each $z^{\beta}$.
    In the present paper, we focus on genus-0 case, i.e., $\Sigma=\mathbb P^1$. 
    The Bergman kernel of $\mathbb P^1$ is given by
    \[
    B(z_{1},z_{2})\coloneqq\frac{\rd z_{1}\rd z_{2}}{(z_{1}-z_{2})^{2}}.
    \]
    
Near each critical point $z^{\beta}$, we denote by $\bar z$ the local involution of $z$ such that $x(\bar{z})=x(z)$.
The corresponding recursion kernel is
\[ K_{\beta}(z_{0},z) \coloneqq \frac{\int_{z'=\bar{z}}^{z} B(z_{0},z')}{2\left( y(z) - y(\bar{z}) \right) \rd x(z)}, \]
and the initial data are
\[
\omega_{0,1}(z)\coloneqq y(z)\,\rd x(z),
\qquad
\omega_{0,2}(z_{1},z_{2})\coloneqq B(z_{1},z_{2}).
\]
For $2g-2+n+1>0$, the recursion defines
$\omega_{g,n+1}(z_{0},z_{1},\ldots,z_{n})$ by
\begin{equation}\label{equ:EO_recursion}
\begin{aligned}
    \omega_{g,n+1}(z_{0},z_{[n]}) \coloneqq & \sum_{\beta} \operatorname{Res}_{z=z^{\beta}} K_{\beta}(z_{0},z)\bigg(\omega_{g-1,n+2}(z,\bar{z},z_{[n]}) \\
    &\qquad \qquad +\sum^{\prime}_{\substack{g_{1}+g_{2}=g\\ I\sqcup J=[n]}}\omega_{g_{1},|I|+1}(z,z_{I}) \omega_{g_{2},|J|+1}(\bar{z},z_{J})\bigg).
\end{aligned}
\end{equation}
Here $[n]\coloneqq\{1,\ldots,n\}$, $z_{I}\coloneqq(z_{i})_{i\in I}$, and the symbol $\sum\limits^{\prime}$ means we exclude $\omega_{0,1}$ in the summation.
For $g\geq 2$, the genus-$g$ free energy is
\begin{equation} \omega_{g,0} \coloneqq \frac{1}{2-2g} \sum_{\beta} \operatorname{Res}_{z=z^{\beta}} \omega_{g,1}(z) \int_{z'=z^{\beta}}^{z} \omega_{0,1}(z'). 
\end{equation}

We next recall how a semisimple CohFT is constructed from the spectral curve data $\mathcal{C} = (\bbP^{1},x,y,B)$.
Around each critical point $z^{\beta}$, we define the {local Airy coordinates} $\eta^{\beta} = \eta^{\beta}(z)$ by
\[ x(z) = x^{\beta} + \frac{1}{2}\eta^{\beta}(z)^{2}, \]
where $x^{\beta} \coloneqq x(z^{\beta})$.
Let
\[ \rd\zeta^{\bar{\beta}}(z) \coloneqq -\operatorname{Res}_{z'=z^{\beta}} \frac{B(z',z)}{\eta^{\beta}(z')}. \]
The state space associated with $\mathcal{C}$ is $V \coloneqq \operatorname{span}\{ \bar{e}_{\beta} \}_{\beta=1}^{N}$, with pairing $\eta(\bar{e}_{\beta}, \bar{e}_{\gamma}) \coloneqq \delta_{\beta,\gamma}$.
The $R$-matrix and $T$-vector associated with $\mathcal{C}$ are defined by
\begin{equation}\label{equ:Rmat_TR}
    \frac{\fu}{\sqrt{2\pi\mathfrak{u}}} \int_{\mathfrak{L}_{\gamma}} \ee^{-x(z)/\mathfrak{u}} \rd \zeta^{\bar{\beta}}(z) \asymp \ee^{-x^{\gamma}/\mathfrak{u}} \eta(\bar{e}_{\gamma}, R^{*}(-\mathfrak{u})\bar{e}_{\beta} ),
\end{equation}
and
\[ \frac{\fu}{\sqrt{2\pi\mathfrak{u}}} \int_{\mathfrak{L}_{\gamma}} \ee^{-x(z)/\mathfrak{u}} \rd y(z) \asymp \ee^{-x^{\gamma}/\mathfrak{u}} \left( \mathfrak{u} - \eta( \bar{e}_{\gamma}, T(\mathfrak{u}) ) \right). \]
Here $R^{*}$ denotes the adjoint with respect to $\eta$, and $\mathfrak{L}_{\gamma}$ is the Lefschetz thimble through $z^{\gamma}$, along which $x(z)-x^\gamma\in\mathbb R_{\geq0}$.
The $R$-matrix defined above satisfies the symplectic condition~\cite{Eyn14}.
Using the $R$-matrix action and translation described in Section~\ref{subsec:GT_reconstruction}, we define the associated CohFT by
    \[ \Omega^{\mathcal{C}} \coloneqq R\cdot T\cdot\Omega^{\mathrm{pt},N}. \]

For $k\geq 0$, we define
\[ \rd \zeta^{\bar{\beta}}_{k} \coloneqq \left( -\, \rd \circ \frac{1}{\rd x} \right)^{k} \rd\zeta^{\bar{\beta}}. \]
These are globally defined meromorphic $1$-forms on $\bbP^{1}$ with poles at the critical points.
By~\cite{DOSS14,Eyn14}, for $2g-2+n > 0$, we have
\begin{equation}\label{equ:TR_CohFT_open_leaves}
    \omega_{g,n}(z_{1},\ldots,z_{n})
    =\sum_{\substack{k_{1},\ldots,k_{n}\geq 0\\
    \beta_{1},\ldots,\beta_{n}\in[N]}}
    \left(
    \int_{\Mbar_{g,n}}
    \Omega^{\mathcal{C}}_{g,n}(\bar{e}_{\beta_{1}},\ldots,\bar{e}_{\beta_{n}})
    \prod_{i=1}^{n}\psi_{i}^{k_{i}}
    \right)
    \prod_{i=1}^{n}\rd\zeta_{k_{i}}^{\bar{\beta}_{i}}(z_{i}).
\end{equation}
In particular, for $n=0$ and $g\geq 2$, this gives
\begin{equation}\label{equ:TR_CohFT_closed}
\omega_{g,0}=\int_{\Mbar_{g}}\Omega^{\mathcal{C}}_{g,0}.
\end{equation}

\subsection{The spectral curve for \texorpdfstring{$X_{p}$}{Xp}}
\label{subsec:calibrated_spectral_curve_Xp}

We consider the spectral curve
\[
\mathcal{C}_{f}\coloneqq(\bbP^{1},x_{f}(z),y(z)),
\]
where
\begin{equation}\label{equ:SpecCurv}
    \begin{aligned}
        x_{f}(z)
        &\coloneqq\log z-(f+1)\log(1-z)+(pf-1)\log(1-\mathbf{q}z),\\
        y(z)
        &\coloneqq -\log(1-z)+p\log(1-\mathbf{q}z).
    \end{aligned}
\end{equation}
Clearly, $x_f(z)=x_0(z)+f y(z)$.
\begin{remark}
{
    For comparison, let $(\widetilde{x},\widetilde{y})$ denote Eynard's logarithmic spectral functions~\cite[Equation~(3.115)]{Eyn08}.
    Identifying Eynard's parameters $p$ and $z_0$ with
    our $p+1$ and $\mathbf{q}^{-1/2}$, respectively, we have
    \[\textstyle
    \widetilde{x}(\frac{1}{\sqrt{\mathbf{q}}z})=-x_0(z)+c_1,\qquad
    \widetilde{y}(\frac{1}{\sqrt{\mathbf{q}}z})=-y(z)-\frac{p-1}{2}x_0(z)+c_2,
    \]
    where $c_1,c_2$ are independent of $z$.
    Our choice of spectral functions is motivated by Zhou's emergent geometry of the KP hierarchy~\cite{Zhou24}.
}
\end{remark}

To determine the critical points of $x_{f}$, we differentiate \eqref{equ:SpecCurv} and obtain
\begin{equation}\label{equ:xf_prime_Nf}
x_{f}'(z)=\frac{N_{f}(z)}{z(1-z)(1-\mathbf{q}z)},
\end{equation}
where
\begin{equation}\label{equ:gammaf}
N_{f}(z)\coloneqq
1+f(1-p\mathbf{q})z+\mathbf{q} \gamma_{f}z^{2},
\qquad
\gamma_{f}\coloneqq f(p-1)-1.
\end{equation}
The two critical points are\footnote{The square root here is chosen to be $f+O(\mathbf{q})$.}
\[
    a_{\alpha} =
    \frac{-f(1-p\mathbf{q})+(-1)^{\alpha}\sqrt{f^{2}(1-p\mathbf{q})^{2}-4\mathbf{q} \gamma_{f}}}
    {2\mathbf{q} \gamma_{f}},
    \qquad \alpha=0,1.
\]
As $\mathbf{q}\to 0$, the critical points satisfy
\[
a_{0}=-\frac{1}{f}+O(\mathbf{q}),\qquad
\mathbf{q} a_{1}=-\frac{f}{\gamma_{f}}+O(\mathbf{q}).
\]
We will therefore use $w\coloneqq\mathbf{q}z$ near $a_{1}$.
Note that the Bergman kernel has the same expression in this coordinate.

To relate the critical points to the eigenvalues of the quantum product \eqref{equ:quantum_product_H}, we compute
\[
D x_{f}(z)
=
\frac{\mathbf{q}(1-\mathbf{q})(1-pf)z}
{(1-p^{2}\mathbf{q})(1-\mathbf{q}z)},
\]
and set
\[
K_{f}\coloneqq\frac{f}{2}+\frac{(p+1)\mathbf{q}}{1-p^{2}\mathbf{q}}.
\]
Using $N_{f}(a_{\alpha})=0$, we find that $K_{f}-Dx_{f}(a_{\alpha})$ satisfies the characteristic equation of \eqref{equ:quantum_product_H}, and its limit as $\mathbf{q}\to0$ is $(-1)^{\alpha}f/2$, so
\begin{equation}\label{equ:eigenvalue_critical_point_Xp}
L_{\alpha}^{f}=K_{f}-Dx_{f}(a_{\alpha}).
\end{equation}
Substituting~\eqref{equ:eigenvalue_critical_point_Xp} into the expression for $\Delta_{\alpha}^{f}$ in Section~\ref{sec:formal_open_closed} gives
\[ \frac{y'(a_{\alpha})^{2}}{x_{f}''(a_{\alpha})}
=(\Delta_{\alpha}^{f})^{-1}. \]
Using the same square roots $(\Delta_\alpha^f)^{1/2}$ as in the definition of $\bar e_\alpha^f$, we choose $\sqrt{x_f''(a_\alpha)}$ so that
\begin{equation}\label{equ:Airy_sign_Xp}
\frac{y'(a_{\alpha})}{\sqrt{x_{f}''(a_{\alpha})}}
=(\Delta_{\alpha}^{f})^{-1/2}.
\end{equation}
This fixes the sign of the local Airy coordinate $\eta^{\alpha}$ by requiring
\[
(\eta^{\alpha})'(a_{\alpha})
=\sqrt{x_{f}''(a_{\alpha})}.
\]
The definition of $\rd\zeta^{\bar{\alpha}}$ then gives
\begin{equation}\label{equ:basic_dzeta_Xp}
\rd\zeta^{\bar{\alpha}}(z)
=\frac{1}{\sqrt{x_{f}''(a_{\alpha})}}
\rd\frac{1}{z-a_{\alpha}}.
\end{equation}

Let $V_{\rB}$ be the state space associated with $\mathcal{C}_{f}$, with normalized canonical basis $\{\bar{e}_{\alpha}^{\rB}\}_{\alpha=0,1}$.
Its canonical idempotents are
\[
e_{\alpha}^{\rB}
\coloneqq(\Delta_{\alpha}^{f})^{-1/2}\bar{e}_{\alpha}^{\rB},\qquad \alpha = 0,1.
\]
The map
\begin{equation}\label{equ:state_space_identification}
\iota_{f}:V_{\rB}\longrightarrow V_{f},
\qquad
\iota_{f}(\bar{e}_{\alpha}^{\rB})\coloneqq\bar{e}_{\alpha}^{f}
\end{equation}
preserves the unit, product, and pairing.
We henceforth identify $V_{\rB}$ with $V_{f}$ via $\iota_{f}$.
In particular,
\[
\sum_{\alpha=0}^{1}e_{\alpha}^{\rB}=\mathbf{1},
\qquad
\sum_{\alpha=0}^{1}L_{\alpha}^{f}e_{\alpha}^{\rB}=H.
\]

Let $\check{R}^{f}$ and $\check{T}$ denote the $R$-matrix and $T$-vector associated with $\mathcal{C}_{f}$, respectively.
To verify the flat unit condition, we define
\begin{equation}\label{equ:flat_unit_partial_fraction}
\varphi_{0}(z)
\coloneqq\sum_{\alpha}
\frac{y'(a_{\alpha})}{x_{f}''(a_{\alpha})}\frac{1}{z-a_{\alpha}}.
\end{equation}
Equations \eqref{equ:Airy_sign_Xp} and \eqref{equ:basic_dzeta_Xp} then imply
\[
\rd\varphi_{0}=\sum_{\alpha}(\Delta_{\alpha}^{f})^{-1/2}\rd\zeta^{\bar{\alpha}}.
\]
Comparing the residues at the two zeros of $N_{f}$ and the value at infinity gives
\begin{equation}\label{equ:varphi0_curve}
\varphi_{0}(z)
=\frac{(1-p)+(p\mathbf{q}-1)z}{\gamma_{f}N_{f}(z)}
=\frac{y'(z)}{x_{f}'(z)}-\frac{p-1}{\gamma_{f}}.
\end{equation}
Integration by parts then yields
\[
\int_{\mathfrak{L}_{\gamma}}\ee^{-x_{f}/\fu}\rd y
=\fu\int_{\mathfrak{L}_{\gamma}}
\ee^{-x_{f}/\fu}\rd\varphi_{0}.
\]
Substituting this identity into the definitions of $\check{R}^f$ and $\check T$ and using the symplectic condition, we obtain
\begin{equation}\label{equ:B_translation_Xp}
\check{T}(\fu)
=\fu\left(\bar{\mathbf{1}}
-(\check{R}^{f})^{-1}(\fu)\mathbf{1}\right).
\end{equation}

\subsection{Closed mirror symmetry}
\label{subsec:R_identification_Xp}

We identify the two $R$-matrices by comparing their quantum differential equations and initial values at $\mathbf{q}=0$, and then deduce closed mirror symmetry.

\begin{proposition}\label{prop:identification_R}
The A-model and B-model $R$-matrices satisfy
\[ R^{f}(z)=\check{R}^{f}(-z). \]
\end{proposition}

\begin{proof}
By the characterization in Section~\ref{subsec:GT_reconstruction}, 
it suffices to show that, under the identification\eqref{equ:state_space_identification}, 
$\check R^f(-z)$ satisfies the same quantum differential equations and initial conditions as $R^f(z)$.

We first verify the quantum differential equations.
To compute the pairing with $H$ from \eqref{equ:Rmat_TR}, we extend the definition of $\varphi_{0}$ by setting
\begin{equation}\label{equ:varphi_i_partial_fraction}
\varphi_{i}(z)
\coloneqq\sum_{\alpha}
(L_{\alpha}^{f})^{i}
\frac{y'(a_{\alpha})}{x_{f}''(a_{\alpha})}
\frac{1}{z-a_{\alpha}},
\qquad i=0,1.
\end{equation}
Equations \eqref{equ:Airy_sign_Xp} and \eqref{equ:basic_dzeta_Xp} then imply
\[
\rd\varphi_{i}
=\sum_{\alpha} (L_{\alpha}^{f})^{i}(\Delta_{\alpha}^{f})^{-1/2}
\rd\zeta^{\bar{\alpha}}.
\]
Comparing the residues at $a_{0},a_{1}$ and the value at infinity gives
\begin{equation}\label{equ:varphi1_curve}
\varphi_{1}(z)
=\frac{2-f(p-1)+
\bigl(f+\mathbf{q}(f(p-2)-2)\bigr)z}
{2\gamma_{f}N_{f}(z)}.
\end{equation}
We write the quantum product \eqref{equ:quantum_product_H} as
$H\ast_{f}H=A_{f}\mathbf{1}+B_{f}H$, where
\[
A_{f}
\coloneqq\frac{f^{2}(1-\mathbf{q})+(2-(p-1)f)^{2}\mathbf{q}}
{4(1-p^{2}\mathbf{q})},
\qquad
B_{f}
\coloneqq\frac{p(2-(p-1)f)\mathbf{q}}{1-p^{2}\mathbf{q}}.
\]
As in Section~\ref{sec:formal_open_closed}, the components of $\check{R}^{f}(\fu)^{*}H^{i}$ in the normalized canonical basis are denoted by
\[
\check{R}^{f}_{i\bar{\gamma}}(\fu)
\coloneqq
\left\langle
\check{R}^{f}(\fu)^{*} H^{i},\bar{e}_{\gamma}^{f}
\right\rangle_{f},
\qquad i=0,1.
\]
Equations \eqref{equ:Rmat_TR} and \eqref{equ:varphi_i_partial_fraction} give
\begin{equation}\label{equ:Rcheck_integral_components}
\frac{\fu\ee^{x_{f}(a_{\gamma})/\fu}}{\sqrt{2\pi\fu}}
\int_{\mathfrak{L}_{\gamma}}
\ee^{-x_{f}(z)/\fu}\rd\varphi_{i}(z)
\asymp\check{R}^{f}_{i\bar{\gamma}}(-\fu).
\end{equation}
For $\varphi=\varphi_{0},\varphi_{1}$, integration by parts gives
\[
\fu D\int_{\mathfrak{L}_{\gamma}}
\ee^{-x_{f}/\fu}\rd\varphi
=
\int_{\mathfrak{L}_{\gamma}}\ee^{-x_{f}/\fu}
\bigl((D\varphi)\rd x_{f}-(Dx_{f})\rd\varphi\bigr).
\]
To evaluate the right-hand side, we differentiate the expressions \eqref{equ:varphi0_curve} and \eqref{equ:varphi1_curve}, obtaining
\begin{equation}\label{equ:GM_rational_identities}
\begin{aligned}
(D\varphi_{0})\,\rd x_{f}-(Dx_{f})\,\rd\varphi_{0}
&=-K_{f}\,\rd\varphi_{0}+\rd\varphi_{1},\\
(D\varphi_{1})\,\rd x_{f}-(Dx_{f})\,\rd\varphi_{1}
&=A_{f}\,\rd\varphi_{0}
+(B_{f}-K_{f})\,\rd\varphi_{1}.
\end{aligned}
\end{equation}
Since $x_f'(a_\gamma)=0$, Equation~\eqref{equ:eigenvalue_critical_point_Xp} gives
\[
D\bigl[x_f(a_\gamma)\bigr]
=(Dx_f)(a_\gamma)=K_f-L_\gamma^f.
\]
Differentiating \eqref{equ:Rcheck_integral_components} and using \eqref{equ:GM_rational_identities}, we therefore obtain, for $\gamma=0,1$,
\begin{equation}\label{equ:Rcheck_QDE}
\begin{aligned}
(\fu D+L_{\gamma}^{f})\check{R}^{f}_{0\bar{\gamma}}(-\fu)
&=\check{R}^{f}_{1\bar{\gamma}}(-\fu),\\
(\fu D+L_{\gamma}^{f})\check{R}^{f}_{1\bar{\gamma}}(-\fu)
&=A_{f}\check{R}^{f}_{0\bar{\gamma}}(-\fu)
+B_{f}\check{R}^{f}_{1\bar{\gamma}}(-\fu).
\end{aligned}
\end{equation}
These are precisely the quantum differential equations satisfied by $R^f(\fu)$.

It remains to compare the initial values at $\mathbf{q}=0$. We use the components
\[
(\check{R}^{f})_{i}^{\ \alpha}(\fu)
\coloneqq
\left\langle
\check{R}^{f}(\fu)^{*} H^{i},(e^{f})^{\alpha}
\right\rangle_{f}
=(\Delta_{\alpha}^{f})^{1/2}\check{R}^{f}_{i\bar{\alpha}}(\fu).
\]
We will show that
\begin{equation}\label{equ:Rcheck_initial_value}
\left.
(\check{R}^{f})_{0}^{\ \alpha}(-\fu)
\right|_{\mathbf{q}=0}
=\mathcal{R}_{\alpha}^{f}(\fu),\qquad \alpha = 0,1.
\end{equation}
By \eqref{equ:Rmat_TR}, \eqref{equ:flat_unit_partial_fraction} and \eqref{equ:varphi0_curve}, integration by parts gives
\begin{equation}\label{equ:Rcheck_initial_integral}
\frac{(\Delta_{\alpha}^{f})^{1/2}}{\sqrt{2\pi\fu}}
\int_{\mathfrak{L}_{\alpha}}
\ee^{-(x_{f}-x_{f}(a_{\alpha}))/\fu}\rd y
\asymp (\check{R}^{f})_{0}^{\ \alpha}(-\fu).
\end{equation}
We take $\mathbf{q}\to 0$ in the coordinate $z$ for $\alpha=0$ and $w=\mathbf{q}z$ for $\alpha=1$.
In these coordinates, $x_{f}-x_{f}(a_{\alpha})$ and $\rd y$ extend holomorphically to $\mathbf{q}=0$ near the critical points, which remain nondegenerate for generic $p,f$.
For each critical point, we first work in a real parameter range where the limiting integral converges.
For $\alpha=0$, we take $-1<f<0$ and $\fu>0$.
At $\mathbf{q}=0$, set $s=1/z$. Then
\[
x_f(1/s)-x_f(a_0)
=f\log\frac{s}{-f}-(f+1)\log\frac{1-s}{1+f},
\qquad
\rd y=\frac{\rd s}{s(1-s)}.
\]
We choose the square root
$\left.(\Delta_0^f)^{-1/2}\right|_{\mathbf{q}=0}
=1/\sqrt{-f(f+1)}$.
Then \eqref{equ:Airy_sign_Xp} orients the limiting
Lefschetz thimble from $s=0$ to $s=1$. Thus
\[
\int_0^1
s^{-f/\fu-1}(1-s)^{(f+1)/\fu-1}\rd s
=
\frac{\Gamma(-f\fu^{-1})\Gamma((f+1)\fu^{-1})}
{\Gamma(\fu^{-1})}.
\]
For $\alpha=1$, we take $0<f<1/p$ and $\fu>0$.
As $\mathbf{q}\to0$, the critical point satisfies
$\mathbf{q}a_1\to f/(1-(p-1)f)$ and
\[
x_f(w/\mathbf{q})-x_f(a_1)
\longrightarrow{}
-f\log\frac{\bigl(1-(p-1)f\bigr)w}{f}
-(1-pf)\log
\frac{\bigl(1-(p-1)f\bigr)(1-w)}{1-pf},
\]
while
\[
\rd y\longrightarrow
-\left(\frac{1}{w}+\frac{p}{1-w}\right)\rd w.
\]
We choose the square root
\[
\left.(\Delta_1^f)^{-1/2}\right|_{\mathbf{q}=0}
=\frac{1}{\sqrt{f(1-pf)(1-(p-1)f)}}.
\]
Then \eqref{equ:Airy_sign_Xp} orients the limiting
Lefschetz thimble from $w=1$ to $w=0$. Thus
\[
\int_1^0
w^{f/\fu}(1-w)^{(1-pf)/\fu}
\left(-\frac{1}{w}-\frac{p}{1-w}\right)\rd w=\frac{1}{1-(p-1)f}
\frac{\Gamma(f\fu^{-1})\Gamma((1-pf)\fu^{-1})}
{\Gamma((1-(p-1)f)\fu^{-1})}.
\]
Substituting these evaluations into \eqref{equ:Rcheck_initial_integral} and applying Stirling's formula, we obtain
\[
\left.(\check{R}^{f})_{0}^{\ \alpha}(-\fu)
\right|_{\mathbf{q}=0}
=
\exp\left(
-\sum_{k\geq 1}
\frac{B_{2k}N^{f}_{2k-1,\alpha}}
{2k(2k-1)}\fu^{2k-1}
\right)
=\mathcal{R}_{\alpha}^{f}(\fu).
\]
The coefficients of these series are rational functions of $p,f$, so the identities extend coefficientwise to generic $p,f$.

Therefore, $\check R^f(-z)$ and $R^f(z)$ satisfy the same quantum differential equations and initial conditions.
The desired identity follows from the uniqueness of the solution.
\end{proof}

We denote by $\check{\Omega}^{f}$ the CohFT associated with $\mathcal{C}_{f}$.
For $g\geq 2$, we set
\[
\check{\mathcal{F}}_{g}(\mathbf{q}) \coloneqq \int_{\Mbar_{g}}\check{\Omega}^{f}_{g,0}.
\]

\begin{theorem}\label{thm:FgAB}
For $g\geq 2$, under the mirror map \eqref{equ:mirror_map}, we have
\[
\mathcal{F}_{g}(\mathbf{Q})=(-1)^{g-1}\check{\mathcal{F}}_{g}(\mathbf{q}).
\]
\end{theorem}

\begin{proof}
The vertex contributions of the underlying TFTs agree under \eqref{equ:state_space_identification}.
By Proposition~\ref{prop:identification_R} and \eqref{equ:B_translation_Xp}, it remains to compare the signs in the two graph sums under $\fu=-z$.
A leg contribution of degree $k$ in the $\psi$-classes is multiplied by $(-1)^{k}$.
An edge term of degree $k+\ell$ in the $\psi$-classes is multiplied by $(-1)^{k+\ell+1}$ and has codimension $k+\ell+1$ after pushforward along the gluing map.
A translation term of degree $m$ is multiplied by $(-1)^{m-1}$ and has codimension $m-1$ after pushforward along the forgetful map.
Hence every graph contribution of codimension $d$ is multiplied by $(-1)^{d}$.
Since $\dim \Mbar_{g}=3g-3$, integration gives the sign $(-1)^{3g-3}=(-1)^{g-1}$.
\end{proof}

\subsection{Open mirror symmetry}
\label{subsec:open_mirror_symmetry_Xp}

Let $f\in\bbZ$. The B-model open coordinate is
\begin{equation}\label{equ:open_coordinate_map_Xp}
\widehat{X}\coloneqq\ee^{x_{f}(z)}
=z(1-z)^{-f-1}(1-\mathbf{q}z)^{pf-1}.
\end{equation}
It is related to the A-model open coordinate $X$ by
\[
\widehat{X} = -\frac{X}{(1-\mathbf{q})^{p+1}}.
\]
We denote by $\rho$ the local inverse of $\widehat{X}$ at $z=0$.
The unstable B-model potentials are normalized by
\begin{equation}\label{equ:B_disk_annulus_definitions}
\begin{aligned}
\check{\mathcal{F}}_{0,1}^{f}(\mathbf{q};\widehat{X})
&\coloneqq\int_{U=0}^{\widehat{X}}y(\rho(U))
\frac{\rd U}{U},\\
\check{\mathcal{F}}_{0,2}^{f}(\mathbf{q};\widehat{X}_{1},\widehat{X}_{2})
&\coloneqq\int_{U_{1}=0}^{\widehat{X}_{1}}
\int_{U_{2}=0}^{\widehat{X}_{2}}
\left(
(\rho\times\rho)^{*}B
-\frac{\rd U_{1}\rd U_{2}}{(U_{1}-U_{2})^{2}}
\right).
\end{aligned}
\end{equation}

For $k\geq 0$ and $\beta=0,1$, we define
\[
\zeta_{k}^{\bar{\beta}}(\widehat{X})
\coloneqq
\int_{0}^{\widehat{X}}\rho^{*}\rd\zeta_{k}^{\bar{\beta}}.
\]
For $n>0$ and $2g-2+n>0$, the B-model potentials are obtained by integrating the Eynard--Orantin differentials $\omega_{g,n}^{f}$ of $\mathcal{C}_{f}$ in each open coordinate from $0$ to $\widehat{X}_{i}$.
Equation~\eqref{equ:TR_CohFT_open_leaves} then gives
\begin{equation}\label{equ:TR_integrated_open_leaves}
\check{\mathcal{F}}_{g,n}^{f}(\mathbf{q};\widehat{X}_{1},\ldots,\widehat{X}_{n})
=
\sum_{\substack{k_{1},\ldots,k_{n}\geq 0\\
\beta_{1},\ldots,\beta_{n}\in\{0,1\}}}
\left(
\int_{\Mbar_{g,n}}
\check{\Omega}^{f}_{g,n}
\bigl(
\bar{e}_{\beta_{1}}^{f},\ldots,\bar{e}_{\beta_{n}}^{f}
\bigr)
\prod_{i=1}^{n}\psi_{i}^{k_{i}}
\right)
\prod_{i=1}^{n}\zeta_{k_{i}}^{\bar{\beta}_{i}}(\widehat{X}_{i}).
\end{equation}

We begin with disk mirror symmetry, which will also
be used to compare the leg contributions.

\begin{proposition}\label{prop:disk_mirror_Xp}
For every integer framing $f$, we have
\[
{\mathcal{F}_{0,1}^{f}}(\mathbf{Q};X) = \check{\mathcal{F}}_{0,1}^{f}(\mathbf{q};\widehat{X}).
\]
\end{proposition}

\begin{proof}
We fix $w\geq 1$ and first assume that $f$ is a nonexceptional integral framing.
{On the A-model side, the string equation and~\eqref{equ:degree_zero_open_disk} give}
\begin{equation}\label{equ:A_disk_J0_winding}
[X^{w}]{\mathcal{F}_{0,1}^{f}}(\mathbf{Q};X)
=N_{0,0,(w)}^{f}
J_{0}^{f}\left(\mathbf{Q},\frac{1}{w}\right),
\end{equation}
where $J_{0}^{f}(\mathbf{Q},z)\coloneqq\langle\mathbf{1},S^{f}(z)\Phi_{0}^{f}\rangle_{f}$.
The divisor equation and \eqref{equ:quantum_product_H} imply the quantum differential equation
\begin{equation}\label{equ:disk_scalar_QDE}
\left[
\left(\frac{D}{w}+\frac{f}{2}\right)^{2}
-B_{f}\left(\frac{D}{w}+\frac{f}{2}\right)-A_{f}
\right]
J_{0}^{f}\left(\mathbf{Q},\frac{1}{w}\right)=0,
\qquad
J_{0}^{f}\left(0,\frac{1}{w}\right)=1.
\end{equation}
Regarding $f$ as a formal parameter, Equation \eqref{equ:disk_scalar_QDE} uniquely determines $J_{0}^{f}(\mathbf{Q},1/w)$ as a formal power series in $\mathbf{Q}$ with coefficients rational in $f$.

On the B-model side, Lagrange inversion applied to \eqref{equ:open_coordinate_map_Xp} gives
\begin{equation}\label{equ:B_disk_lagrange_coefficients}
[\widehat{X}^{w}]\check{\mathcal{F}}_{0,1}^{f}(\mathbf{q};\widehat{X})
=\frac{(-1)^{w+1}}{w^{2}}
\sum_{j=0}^{w}C_{w,j}^{f}\mathbf{q}^{j},
\end{equation}
where
\[
C_{w,j}^{f}
\coloneqq\binom{(1-pf)w}{j}\binom{(f+1)w-1}{w-1-j}
-p\binom{(1-pf)w-1}{j-1}\binom{(f+1)w}{w-j}.
\]
Here $\binom{a}{k}=0$ for $k<0$.
To compare the B-model disk coefficients with $J_{0}^{f}(\mathbf{Q},1/w)$, we define
\[
P_{w}^{f}(\mathbf{q})
\coloneqq\sum_{j=0}^{w}
\frac{C_{w,j}^{f}}{C_{w,0}^{f}}\mathbf{q}^{j}.
\]
Substituting $\widehat{X}=-X/(1-\mathbf{q})^{p+1}$ in \eqref{equ:B_disk_lagrange_coefficients} and using $C_{w,0}^{f}=-w^{2}N_{0,0,(w)}^{f}$ gives
\begin{equation}\label{equ:B_disk_normalized_solution}
\frac{1}{N_{0,0,(w)}^{f}}
[X^{w}]\check{\mathcal{F}}_{0,1}^{f}(\mathbf{q};\widehat{X})
=(1-\mathbf{q})^{-(p+1)w}P_{w}^{f}(\mathbf{q}).
\end{equation}
Since $P_{w}^{f}(0)=1$, it remains to show that $(1-\mathbf{q})^{-(p+1)w}P_{w}^{f}(\mathbf{q})$ satisfies \eqref{equ:disk_scalar_QDE}.
In terms of $P=P_{w}^{f}$, this equation becomes
\begin{equation}\label{equ:disk_Pw_differential_equation}
\begin{aligned}
0={}&\mathbf{q}(\mathbf{q}-1)(p^{2}\mathbf{q}-1)P''\\
&+\bigl(
(fp^{3}w-2p^{2}w+p^{2})\mathbf{q}^{2}
+(-fp^{2}w-fpw+2w-2)\mathbf{q}
+fw+1
\bigr)P'\\
&+\bigl(
(-fp^{3}w^{2}+p^{2}w^{2})\mathbf{q}
+fpw^{2}+pw-w^{2}+w
\bigr)P.
\end{aligned}
\end{equation}
The identity
\[
\frac{C_{w,j+1}^{f}}{C_{w,j}^{f}}=
\frac{(1-pf)w-(p+1)(j+1)}
{(1-pf)w-(p+1)j}
\frac{(1-pf)w-j}{j+1}
\frac{w-j}{fw+j+1}
\]
implies the three-term recurrence obtained by comparing coefficients in \eqref{equ:disk_Pw_differential_equation}.
By uniqueness,
\[
(1-\mathbf{q})^{-(p+1)w}P_{w}^{f}(\mathbf{q})
=J_{0}^{f}(\mathbf{Q},1/w).
\]
Since $w\geq 1$ is arbitrary, \eqref{equ:A_disk_J0_winding} and \eqref{equ:B_disk_normalized_solution} imply disk mirror symmetry at every nonexceptional integral framing.

It remains to extend the identity to every integer framing.
In positive degree, the coefficients of $\mathcal{F}^{\bbT'}_{0,1}|_{\sfv=f\sfu,\sfu=1}$ are rational in $f$, whereas the B-model coefficients are polynomial in $f$ by \eqref{equ:B_disk_lagrange_coefficients}.
Their equality at all nonexceptional integral framings therefore implies equality as rational functions of $f$.
{By the coefficientwise regularity described in Section~\ref{subsec:open_leaves} and~\eqref{equ:formal_relative_localization}, this identity yields disk mirror symmetry in positive degree at every integer framing.}
In degree zero, substituting $\widehat{X}=-X$ and $C_{w,0}^{f}=-w^{2}N_{0,0,(w)}^{f}$ into \eqref{equ:B_disk_lagrange_coefficients} gives the prescribed A-model disk $\sum_{w\geq 1}N_{0,0,(w)}^{f}X^{w}$.
\end{proof}

The B-model leg contribution is
\begin{equation}\label{equ:B_model_open_leaf}
\mathcal{L}_{\rB}^{f}(\mathbf{q};\widehat{X},\fu)
\coloneqq
(\check{R}^{f}(\fu))^{-1}
\mathcal{U}_{\rB}^{f}(\mathbf{q};\widehat{X},\fu),
\end{equation}
where
\[
\mathcal{U}_{\rB}^{f}(\mathbf{q};\widehat{X},\fu)
\coloneqq\sum_{\beta=0}^{1}\sum_{k\geq 0}
\zeta_{k}^{\bar{\beta}}(\widehat{X})\bar{e}_{\beta}^{f}\fu^{k}.
\]
To compare the leg contributions at each winding number, we define
\[
W_{w}^{\rB}\coloneqq[\widehat{X}^{w}]\mathcal{U}_{\rB}^{f}(\mathbf{q};\widehat{X},0),\qquad w\geq 1.
\]

\begin{proposition}\label{prop:open_leaf_identity_Xp}
Let $f$ be a nonexceptional integral framing.
Then
\[
\mathcal{L}_{\rA}^{f}(\mathbf{Q};X,z) = \mathcal{L}_{\rB}^{f}(\mathbf{q};\widehat{X},-z).
\]
\end{proposition}

\begin{proof}
Equations \eqref{equ:flat_unit_partial_fraction}, \eqref{equ:varphi0_curve}, and \eqref{equ:varphi_i_partial_fraction} give
\[
\begin{aligned}
\langle\mathbf{1},\mathcal{U}_{\rB}^{f}(\mathbf{q};\widehat{X},0)\rangle_{f}
&=\varphi_{0}(\rho(\widehat{X}))-\varphi_{0}(0)
=\frac{\partial y}{\partial x_{f}},\\
\langle H,\mathcal{U}_{\rB}^{f}(\mathbf{q};\widehat{X},0)\rangle_{f}
&=\varphi_{1}(\rho(\widehat{X}))-\varphi_{1}(0).
\end{aligned}
\]
Recall that $V_{w}^{f}=w^{2}N_{0,0,(w)}^{f}S^{f}(1/w)\Phi_{0}^{f}$.
Proposition~\ref{prop:disk_mirror_Xp} and \eqref{equ:A_disk_J0_winding} imply
\[
\langle\mathbf{1},\frac{(-1)^{w}W_{w}^{\rB}}{(1-\mathbf{q})^{(p+1)w}}\rangle_{f}
=w^{2}[X^{w}]\check{\mathcal{F}}_{0,1}^{f}(\mathbf{q};\widehat{X})
=\langle\mathbf{1},V_{w}^{f}\rangle_{f}.
\]

Expressing the first identity in \eqref{equ:GM_rational_identities} in the coordinate $\widehat{X}$ and comparing coefficients gives
\begin{equation}\label{equ:H_pairing_B_leaf}
\langle H,W_{w}^{\rB}\rangle_{f}=\left(K_{f}+\frac{D}{w}\right)\langle\mathbf{1},W_{w}^{\rB}\rangle_{f}.
\end{equation}
Since
\[
D\bigl((1-\mathbf{q})^{-(p+1)w}\bigr)
=w\left(K_{f}-\frac{f}{2}\right)(1-\mathbf{q})^{-(p+1)w},
\]
Equation \eqref{equ:H_pairing_B_leaf} becomes
\[
\langle H,\frac{(-1)^{w}W_{w}^{\rB}}{(1-\mathbf{q})^{(p+1)w}}\rangle_{f}
=\left(\frac{f}{2}+\frac{D}{w}\right)
\langle\mathbf{1},\frac{(-1)^{w}W_{w}^{\rB}}{(1-\mathbf{q})^{(p+1)w}}\rangle_{f}.
\]
The divisor equation gives the same relation for $V_{w}^{f}$.
Their pairings with $\mathbf{1}$ and $H$ therefore agree, and nondegeneracy of the pairing yields
\begin{equation}\label{equ:vector_winding_identity}
\frac{(-1)^{w}W_{w}^{\rB}}{(1-\mathbf{q})^{(p+1)w}}=V_{w}^{f}.
\end{equation}

Since $\rd\zeta_{k}^{\bar{\beta}}=(-\rd\circ 1/\rd x_{f})^{k}\rd\zeta^{\bar{\beta}}$, we have
\begin{equation}\label{equ:zeta_winding_coefficients}
[\widehat{X}^{w}]\zeta_{k}^{\bar{\beta}}
=(-w)^{k}[\widehat{X}^{w}]\zeta_{0}^{\bar{\beta}}.
\end{equation}
Using \eqref{equ:vector_winding_identity} and \eqref{equ:A_model_open_input_expansion}, we obtain
\[
\mathcal{U}_{\rB}^{f}(\mathbf{q};\widehat{X},-z)
=\sum_{w\geq 1}\frac{W_{w}^{\rB}\widehat{X}^{w}}{1-wz}
=\sum_{w\geq 1}\frac{V_{w}^{f}X^{w}}{1-wz}
=\mathcal{U}_{\rA}^{f}(\mathbf{Q};X,z).
\]
Applying Proposition~\ref{prop:identification_R} then proves the proposition.
\end{proof}

We now use this comparison to prove annulus mirror symmetry.

\begin{proposition}\label{prop:annulus_mirror_Xp}
For a nonexceptional integral framing,
\[
{\mathcal{F}_{0,2}^{f}}(\mathbf{Q};X_{1},X_{2})
=-\check{\mathcal{F}}_{0,2}^{f}(\mathbf{q};\widehat{X}_{1},\widehat{X}_{2}).
\]
\end{proposition}

\begin{proof}
We compare the two annulus potentials after applying $X_{1}\partial_{X_{1}}+X_{2}\partial_{X_{2}}$.
For the B-model calculation, since $\widehat{X}_i=-X_i/(1-\mathbf{q})^{p+1}$, we have
\(
X_i\partial_{X_i}
=\widehat{X}_i\partial_{\widehat{X}_i},
\)
$i=1,2$.
We use the identity
\[
\left(\rd_{x_{1}}+\rd_{x_{2}}\right)
\left(\frac{B(z_{1},z_{2})}{\rd x_{1}\rd x_{2}}\right)
\rd x_{1}\rd x_{2}
=-\sum_{\beta=0}^{1}
\rd\zeta^{\bar{\beta}}(z_{1})\rd\zeta^{\bar{\beta}}(z_{2}).
\]
To verify this identity, we compare the poles on both sides, which can occur only along the diagonal or at $z_i=a_\beta$, $i=1,2$.
The polar part of $B/(\rd x_{1}\rd x_{2})$ along the diagonal is $(x_{1}-x_{2})^{-2}$, which is annihilated by
$\partial_{x_{1}}+\partial_{x_{2}}$. At $z_{1}=a_{\beta}$, the polar part of the left-hand side is
\[
-\frac{\rd z_{1}\rd z_{2}}
{x_{f}''(a_{\beta})(z_{1}-a_{\beta})^{2}(z_{2}-a_{\beta})^{2}},
\]
which agrees with that of the right-hand side.
Thus the difference between the two sides is a holomorphic bidifferential on $\bbP^{1}\times\bbP^{1}$, and hence vanishes.
Integrating the identity therefore gives
\[
\left(X_{1}\partial_{X_{1}}+X_{2}\partial_{X_{2}}\right)
\check{\mathcal{F}}_{0,2}^{f}(\mathbf{q};\widehat{X}_{1},\widehat{X}_{2})
=-\sum_{\beta}
\zeta_{0}^{\bar{\beta}}(\widehat{X}_{1})\zeta_{0}^{\bar{\beta}}(\widehat{X}_{2}).
\]

On the A-model side, we apply the two-point descendant formula of \cite[Proposition~2.1 and Remark~2.3]{GT13} to the twisted GW theory of $\bbP^{1}$, specializing its two formal variables to $1/a$ and $1/b$.
Multiplying by the disk factors and applying $X_{1}\partial_{X_{1}}+X_{2}\partial_{X_{2}}$ cancels the factor $1/(a+b)$, {so~\eqref{equ:formal_relative_localization} gives}
\[
\left(X_{1}\partial_{X_{1}}+X_{2}\partial_{X_{2}}\right)
{\mathcal{F}_{0,2}^{f}}(\mathbf{Q};X_{1},X_{2})
=\sum_{a,b\geq 1}
\left\langle V_{a}^{f},V_{b}^{f}\right\rangle_{f}
X_{1}^{a}X_{2}^{b}
\]
in positive degree.
In degree zero, the same identity follows from the unstable convention, since $S^{f}(z)|_{\mathbf{Q}=0}=\operatorname{id}$ and $\langle\Phi_{0}^{f},\Phi_{0}^{f}\rangle_{f}=-f(f+1)$.
Comparing these expressions and using \eqref{equ:vector_winding_identity}, we obtain
\[
\left(X_{1}\partial_{X_{1}}+X_{2}\partial_{X_{2}}\right)
{\mathcal{F}_{0,2}^{f}}(\mathbf{Q};X_{1},X_{2})
=-\left(X_{1}\partial_{X_{1}}+X_{2}\partial_{X_{2}}\right)
\check{\mathcal{F}}_{0,2}^{f}(\mathbf{q};\widehat{X}_{1},\widehat{X}_{2}).
\]
Since $X_{1}\partial_{X_{1}}+X_{2}\partial_{X_{2}}$ multiplies the coefficient of $X_{1}^{a}X_{2}^{b}$ by $a+b\neq 0$ for $a,b\geq 1$, the result follows.
\end{proof}

The disk and annulus identities, together with the comparison of the $R$-matrices and leg contributions, give the open mirror symmetry.

\begin{theorem}\label{thm:open_mirror_Xp}
Let $f\in \bbZ$ be a nonexceptional framing.
Under the mirror map \eqref{equ:mirror_map}, we have, for every $g\geq 0$ and $n\geq 1$,
\[
{\mathcal{F}_{g,n}^{f}}(\mathbf{Q};X_{1},\ldots,X_{n})
=(-1)^{g-1+n}\check{\mathcal{F}}_{g,n}^{f}
(\mathbf{q};\widehat{X}_{1},\ldots,\widehat{X}_{n}).
\]
For $(g,n)=(0,1)$, the identity holds for every integer framing.
\end{theorem}

\begin{proof}
The disk and annulus cases are Propositions~\ref{prop:disk_mirror_Xp} and~\ref{prop:annulus_mirror_Xp}.
For $2g-2+n>0$, Proposition~\ref{prop:formal_relative_open_graph_sum} gives the A-model graph sum.
Applying the graph expansion of $\check{\Omega}^f$ to \eqref{equ:TR_integrated_open_leaves} gives the B-model graph sum.
By Proposition~\ref{prop:open_leaf_identity_Xp}, we have
$\mathcal{L}_{\rB}^{f}(\mathbf{q};\widehat{X},z)
=\mathcal{L}_{\rA}^{f}(\mathbf{Q};X,-z)$.
The same argument as in the proof of Theorem~\ref{thm:FgAB} then gives the sign
$(-1)^{3g-3+n}=(-1)^{g-1+n}$,
since $\dim\Mbar_{g,n}=3g-3+n$.
\end{proof}

\section{Polynomial properties of Gromov--Witten potentials}
\label{sec:proof_main}

In this section, we prove polynomiality results for the closed and open potentials. 
We first establish the required properties of the $R$-matrix and then apply the reconstruction formula.

\subsection{Properties of the \texorpdfstring{$R$}{R}-matrix}
\label{subsec:R_properties_used_later}

{
In this subsection, we deduce some properties of the $R$-matrix that will be used to prove polynomiality of the closed and open GW potentials.
By Proposition~\ref{prop:identification_R}, we have $R^f(z)=\check{R}^f(-z)$, and thus these properties hold equivalently for both the A-model and the B-model $R$-matrices.

We introduce the following notations:
\begin{equation}
R\coloneqq R^{f_{\star}},
\qquad
L\coloneqq\left(\frac{1-\mathbf{q}}{1-p^{2}\mathbf{q}}\right)^{1/2}
=1+O(\mathbf{q}).
\end{equation}
Recall here that $f_{\star}=2/(p-1)$.
We expand $R(z)=\sum_{k\geq 0}R_kz^k$.}
For $d\geq 0$, we denote by $\bbQ[p^{2}][L^{2}]_{\leq d}$ the polynomials in $L^{2}$ of degree at most $d$ with coefficients in $\bbQ[p^{2}]$.

\begin{proposition}\label{prop:R_properties}
{The following properties hold.}
\begin{enumerate}
\item[{(i)}]
Let $f$ be nonexceptional.
Then each coefficient of {$R^{f}$} has a finite limit as $\mathbf{q}\to\infty$.
If $f\neq f_{\star}$, the same holds as $\mathbf{q}\to 1$.
\item[{(ii)}]
{At $f=f_{\star}$, we have}
\[
{R_{0}^{\ 1}(z) = R_{0}^{\ 0}(-z).}
\]
\item[{(iii)}]
{At $f=f_{\star}$, for $k\in\bbZ_{\geq 0}$ and $\alpha=0,1$, we have}
\[
(R_{k})_{i}^{\ \alpha}\in
\frac{\bbQ[p^{2}][L^{2}]_{\leq 2k}}{(p^{2}-1)^{k}\,f_{\star}^{k-i}L^{k-i}},
\qquad i=0,1.
\]
\end{enumerate}
\end{proposition}

\begin{proof}
{We first prove (i) on the B-model side.}
As $\mathbf{q}\to\infty$, the two critical points tend to $0$ and $pf/\gamma_f$, where $\gamma_f$ is defined in \eqref{equ:gammaf}.
We use the coordinate $\zeta=\mathbf{q}z$ near the former and the coordinate $z$ near the latter.
For $f\neq f_{\star}$, as $\mathbf{q}\to1$, the critical points tend to $1/\gamma_f$ and $1$.
These limits are distinct, since $\gamma_f-1=(p-1)(f-f_{\star})\neq0$.
We use the coordinate $z$ near the former and $\zeta=(z-1)/(1-\mathbf{q})$ near the latter.

In both cases, $x_f-x_f(a_\alpha)$ and $\rd y$ extend holomorphically in these coordinates, and the limiting critical points are nondegenerate.
The integral representation~\eqref{equ:Rcheck_initial_integral} then shows that each coefficient of $\check{R}^{f}_{0\bar{\alpha}}$ extends holomorphically and hence has a finite limit.
Equation~\eqref{equ:Rcheck_QDE} gives the same conclusion for $\check{R}^{f}_{1\bar{\alpha}}$.
{Since the normalized canonical basis extends holomorphically and remains a basis at the limit, each coefficient of $\check{R}^{f}$ has a finite limit. Proposition~\ref{prop:identification_R} then gives (i) for $R^{f}$.}

{To prove (ii), we note that at $f=f_{\star}$ the tangent weights at the two fixed points specialize to}
\[
(f_{\star},1,-f_{\star}-1),
\qquad
(-f_{\star},-1,f_{\star}+1).
\]
Note that the two triples are negatives of one another. Hence $N^{f_{\star}}_{2m-1,1}=-N^{f_{\star}}_{2m-1,0}$ and $\mathcal{R}_{1}^{f_{\star}}(z)=\mathcal{R}_{0}^{f_{\star}}(-z)$.
{Since $B_{f_{\star}}=0$ and $L_{1}^{f_{\star}}=-L_{0}^{f_{\star}}$, the QDE \eqref{equ:QDE_R0b} is invariant under $(\alpha,z)\mapsto(1-\alpha,-z)$.
Moreover, Equation~\eqref{equ:Delta_f} gives $\Delta_{1}^{f_{\star}}=-\Delta_{0}^{f_{\star}}$, so the ratio of their chosen square roots is independent of $\mathbf{q}$. Uniqueness of the solution with the initial conditions determined by $\mathcal{R}_{\alpha}^{f_{\star}}$ therefore gives
\[
R_{0\bar{1}}(z)=\frac{(\Delta_{0}^{f_{\star}})^{1/2}}
{(\Delta_{1}^{f_{\star}})^{1/2}}R_{0\bar{0}}(-z).
\]
Multiplying by $(\Delta_{1}^{f_{\star}})^{1/2}$ and using $R_{0}^{\ \alpha}=(\Delta_{\alpha}^{f_{\star}})^{1/2}R_{0\bar{\alpha}}$ then proves (ii).}

{Finally, we prove (iii). At this framing, the coefficient of $H$ in $H\ast_f H$ vanishes.}
The corresponding canonical idempotents and pairing are given by
\[
e_{\alpha}=\frac{1}{2}\left(\mathbf{1}+\frac{2(-1)^{\alpha}}{f_{\star}L}H\right),\qquad \Delta_{\alpha}=(-1)^{\alpha}\frac{(1-p^{2})f_{\star}^{3}L}{4},\qquad \alpha=0,1.
\]
Equation \eqref{equ:QDE_R0b} then becomes
\[
\left((zD+L_{\alpha})^{2}-\frac{f_{\star}^{2}L^{2}}{4}\right)R_{0\bar{\alpha}}=0,
\qquad L_{\alpha}=(-1)^{\alpha}\frac{f_{\star}L}{2}.
\]

We first prove the case $i=0$. {By (ii), we may write}
\[
(R_{k})_{0}^{\ \alpha}
=\frac{r_{k}(L)}{((-1)^{\alpha}f_{\star})^{k}},
\qquad r_{0}=1.
\]
To express $D$ in the variable $L$, we set
\[
A(L)\coloneqq\frac{(L^{2}-1)(p^{2}L^{2}-1)}{4(p^{2}-1)},
\qquad D=2LA(L)\partial_{L}.
\]
Using $D\log\Delta_{\alpha}^{-1/2}=-A(L)$, we expand the QDE to obtain the recursion
\[
r_{k}'=-2A r_{k-1}''-2A'r_{k-1}'
+\left(\frac{A'}{L}-\frac{A}{2L^{2}}\right)r_{k-1}
\]
for $k\geq 1$, where primes denote differentiation with respect to $L$.

We induct on $k$, starting from $r_{0}=1$.
Under the inductive hypothesis, the recursion implies
\[
r_{k}'\in
\frac{\bbQ[p^{2}][L^{2}]_{\leq 2k}}{(p^{2}-1)^{k}L^{k+1}}.
\]
From Proposition~\ref{prop:identification_R} and \eqref{equ:Rmat_TR}, we see that $r_{k}$ is an algebraic function of $L$.
Thus integrating the recurrence produces no logarithmic term, and $r_{k}$ is a Laurent polynomial in $L$ up to an additive constant.
For odd $k$, this constant vanishes: $L\mapsto-L$ interchanges $L_{0}$ and $L_{1}$, and $r_{k}(-L)=(-1)^{k}r_{k}(L)$ by {(ii)}.
For even $k$, the formula for $\mathcal{R}_{0}^{f_{\star}}$ gives $r_{k}(1)=[z^{k}]\mathcal{R}_{0}^{f_{\star}}(f_{\star}z)\in(p^{2}-1)^{-k}\bbQ[p^{2}]$.
Thus
\[
    r_{k}\in\frac{\bbQ[p^{2}][L^{2}]_{\leq 2k}}{(p^{2}-1)^{k}L^{k}}.
\]

For $i=1$, comparing coefficients in $R_{1\bar{\alpha}}=(zD+L_{\alpha})R_{0\bar{\alpha}}$ gives
\[
(R_{k})_{1}^{\ \alpha}
=(D-A)(R_{k-1})_{0}^{\ \alpha}
+(-1)^{\alpha}\frac{f_{\star}L}{2}(R_{k})_{0}^{\ \alpha},
\qquad k\geq 1.
\]
The desired conclusion then follows from the case $i=0$ for $k\geq 1$,
and from $(R_0)_1^{\ \alpha}=L_{\alpha}$ for $k=0$.
\end{proof}

For the reconstruction formula, it is convenient to rewrite this polynomiality in terms of
\[ \Delta = \frac{1}{L^{2}} = \frac{1-p^{2}\mathbf{q}}{1-\mathbf{q}}. \]
We use $\bbQ[p^{2}][\Delta]_{\leq d}$ analogously for polynomials of
degree at most $d$ in $\Delta$.
In terms of $\Delta$, {Proposition~\ref{prop:R_properties}(iii)} reads, for $k\geq 0$,
\begin{equation}\label{equ:Rpoly_K}
 (R_{k})_{i}^{\ \alpha} \in \frac{1}{(p^{2}-1)^{k} \, f_{\star}^{k-i} \Delta^{(3k+i)/2}} \, \bbQ[p^{2}][\Delta]_{\leq 2k},\quad i=0,1.
\end{equation}

\subsection{Polynomiality of the closed potentials}

{Since the closed GW potentials $\mathcal{F}_{g}$ for $g\geq 2$ are independent of $f$, we apply the reconstruction formula to $\Omega^{\star}\coloneqq\Omega^{f_{\star}}$.}
The polynomiality in \eqref{equ:Rpoly_K} yields the following result.

\begin{theorem}\label{thm:Omega}
    For $2g-2+n>0$, $d\in\bbZ_{\geq 0}$, and
    $a_{1},\ldots,a_{n}\in\{0,1\}$, we have
    \[ \left[ \Omega_{g,n}^{\star}(H^{a_{1}},\dots,H^{a_{n}}) \right]_{d} \in \frac{f_{\star}^{3g-3-d+\sum_{i=1}^{n} a_{i}}}{(p^{2}-1)^{d-(g-1)}\, \Delta^{\frac{1}{2}\left( g-1+3d+\sum_{i=1}^{n} a_{i} \right)}} \, \bbQ[p^{2}][\Delta]_{\leq 2d} \otimes H^{2d}(\Mbar_{g,n},\bbQ), \]
    where $[\cdot]_{d}$ denotes the component in
    $H^{2d}(\Mbar_{g,n},\bbQ)$.
\end{theorem}

\begin{proof}
For a genus-$g$, $n$-pointed stable graph $\Gamma$, we consider the vertex, leg, edge, and translation contributions in the reconstruction formula.

\begin{itemize}
\item \emph{Vertices.}
A vertex $v$ with label $\alpha$ contributes
\[
\Delta_{\alpha}^{g_{v}-1}
=(-1)^{\alpha(g_{v}-1)}
\frac{2^{2-2g_{v}}(1-p^{2})^{g_{v}-1}f_{\star}^{3g_{v}-3}}
{\Delta^{(g_{v}-1)/2}}.
\]

\item \emph{Legs.}
The $i$-th leg contributes $R_{a_{i}}^{\ \alpha}(-\psi_{i})$.

\item \emph{Translation.}
Each additional marked point introduced by the translation contributes $\psi\bigl(1-R_{0}^{\ \alpha}(-\psi)\bigr)$.
The coefficient of $\psi^{k+1}$ is $(-1)^{k+1}(R_{k})_{0}^{\ \alpha}$ for $k\geq 1$.

\item \emph{Edges.}
An edge joining vertices with labels $\alpha,\beta$ contributes
\[
\frac{
\delta_{\alpha\beta}\Delta_{\alpha}
-\sum_{i,j=0}^{1}\eta^{ij}
R_{i}^{\ \alpha}(-\psi_{1})R_{j}^{\ \beta}(-\psi_{2})
}{\psi_{1}+\psi_{2}}.
\]
Here the inverse Gram matrix in the flat basis $\{\mathbf{1},H\}$ satisfies $\eta^{00}=\eta^{11}=0$ and $\eta^{01}=\eta^{10}=\frac{(1-p^{2})f_{\star}^{2}}{4}$.
The symplectic condition makes the numerator divisible by $\psi_{1}+\psi_{2}$.
For $k,\ell\geq 0$, the coefficient of $\psi_1^k\psi_2^\ell$ is therefore a linear combination of the coefficients of total degree $k+\ell+1$ in the numerator.
\end{itemize}

Combining \eqref{equ:Rpoly_K} with $\sum_{v\in V(\Gamma)}(g_v-1)+|E(\Gamma)|=g-1$ gives the powers of $f_{\star}$, $\Delta$, and $p^2-1$ in the statement, with numerator in $\bbQ[p^2][\Delta]_{\leq 2d}$. Summing over the labels and stable graphs then proves the theorem.
\end{proof}

Integrating over $\Mbar_g$ and using framing independence and regularity at $\mathbf{q}=1$, we obtain the following result.

\begin{theorem}\label{thm:FgXp}
For $g\geq 2$, the GW potential $\mathcal{F}_{g}$ is the expansion
of a rational function
\begin{equation} 
	\mathcal{F}_{g} = (-1)^{g-1}\frac{\calP_{g}(\Delta,p^{2})}{\Delta^{5g-5}}, 
\end{equation}
where $\calP_{g}(\Delta,p^{2}) \in (p^{2}-1)^{-(2g-2)}\,
\bbQ[p^{2}][\Delta]_{\leq 5g-5}$.
\end{theorem}
\begin{proof}
Applying Theorem~\ref{thm:Omega} with $d=\dim\Mbar_{g}=3g-3$ and $n=0$ gives
\[ \mathcal{F}_{g} \in \frac{1}{(p^{2}-1)^{2g-2} \Delta^{5g-5}}\, \bbQ[p^{2}][\Delta]_{\leq 6g-6}. \]
Framing independence allows us to compute $\mathcal{F}_{g}$ by reconstruction at a nonexceptional $f\neq f_{\star}$.
{Proposition~\ref{prop:R_properties}(i) then implies that it is regular at $\mathbf{q}=1$.}
Since $\Delta\to\infty$ there, the degree of the numerator is at most $5g-5$.
\end{proof}

\subsection{Polynomiality of the open potentials}
\label{subsec:open_polynomiality}

We now study the polynomial structure of the fixed-winding coefficients of the open potentials:
\[
\mathcal{F}^{f}_{g,\boldsymbol{\mu}}(\mathbf{q})
:=[X_{1}^{\mu_{1}}\cdots X_{n}^{\mu_{n}}]
{\mathcal{F}^{f}_{g,n}}(\mathbf{Q};X_{1},\ldots,X_{n}),
\]
where $\mathbf{Q}=\mathbf{q}(1-\mathbf{q})^{p^{2}-1}$.
Using~\eqref{equ:open_vertex_gluing} and~\eqref{equ:open_potential_recovery}, we extend this definition to formal $f$.
For each fixed integer \(p\geq2\), we first establish Laurent polynomiality in \(\Delta\) at \(f=f_\star\) using the reconstruction formula. 
We then extend this result to arbitrary \(f\) by the cut-and-join equation. 
In the final step, B-model regularity and growth estimates yield polynomiality in \(\mathbf q\) after multiplication by explicit factors, with a degree bound, 
while the topological vertex formula establishes polynomial dependence of the coefficients on \(p\) and \(f\).

\begin{proposition}
\label{prop:open_fstar_polynomiality}
{For $n>0$, $2g-2+n>0$, and $\boldsymbol{\mu}\in\bbZ_{>0}^{n}$, we have}
\[
\mathcal{F}^{f_{\star}}_{g,\boldsymbol{\mu}} \in \frac{1}{\Delta^{5g-5+2n}}\bbQ[\Delta].
\]
\end{proposition}

\begin{proof}
{
For nonexceptional integral $f$, the ancestor--descendant relation and~\eqref{equ:formal_relative_localization} give}
\begin{equation}\label{eqn:F-cohft}
\mathcal{F}^{f}_{g,\boldsymbol{\mu}}
=\sum_{k_{1},\ldots,k_{n}\geq 0}
\left(\prod_{i=1}^{n}\mu_{i}^{k_{i}}\right)
\int_{\Mbar_{g,n}}
\Omega^{f}_{g,n}
(V_{\mu_{1}}^{f},\ldots,V_{\mu_{n}}^{f})
\prod_{i=1}^{n}\psi_{i}^{k_{i}}.
\end{equation}
Here $V_{\mu_i}^f$ are the vectors defined in \eqref{equ:A_model_open_input_expansion} and their coefficients are rational in $f$ and regular at $f=f_{\star}$.
Since both sides are coefficientwise rational in $f$ and agree at infinitely many integral values, the identity extends to formal $f$ and specializes to $f=f_{\star}$.

{We show that the vectors {$V_w^{f_\star}$} have polynomial coefficients in $\Delta$ in the flat basis.}
For {$w\geq 1$}, Proposition~\ref{prop:disk_mirror_Xp} and \eqref{equ:B_disk_lagrange_coefficients} imply
\[
A_{w}(\mathbf{q})
\coloneqq\left\langle\mathbf{1},{V_{w}^{f_{\star}}}\right\rangle_{f_{\star}}
=-(1-\mathbf{q})^{-(p+1)w}
\sum_{j=0}^{w} C_{w,j}^{f_{\star}}\mathbf{q}^{j}.
\]
Substituting $\mathbf{q}=(1-\Delta)/(p^{2}-\Delta)$, we obtain {$A_w\in\bbQ[\Delta]$}.
The divisor equation gives the other pairing as
{\[
\left\langle H,{V_{w}^{f_{\star}}}\right\rangle_{f_{\star}}
=\left(\frac{f_{\star}}{2}+\frac{D}{w}\right)A_{w}.
\]}
Evaluating \eqref{equ:disk_Pw_differential_equation} at $f=f_{\star}$ and $\mathbf{q}=p^{-2}$ gives {$\partial_{\Delta}A_w|_{\Delta=0}=0$}.
Together with
\[
D
=-\frac{(1-\Delta)(p^{2}-\Delta)}
{(p^{2}-1)\Delta}\frac{\partial}{\partial\Delta},
\]
this implies {$DA_{w}\in\bbQ[\Delta]$}, and hence
{$\langle H,{V_{w}^{f_{\star}}}\rangle_{f_{\star}}\in\bbQ[\Delta]$}.
Since the pairing is independent of $\mathbf{q}$ and nondegenerate,
the coefficients of {${V_w^{f_\star}}$} in the flat basis belong to $\bbQ[\Delta]$.

Only the component in $H^{2d}(\Mbar_{g,n})$, where $d=3g-3+n-\sum_i k_i\geq0$, contributes to~\eqref{eqn:F-cohft}.
Expanding the insertions in the flat basis and applying Theorem~\ref{thm:Omega}, we find that the pole order at $\Delta=0$ is at most
\[
\frac{g-1+3d+n}{2}
=5g-5+2n-\frac{3}{2}\sum_i k_i
\leq 5g-5+2n. \qedhere
\]
\end{proof}

{We now extend Proposition~\ref{prop:open_fstar_polynomiality} to formal $f$.}

\begin{proposition}\label{prop:open_framing_filtration}
{For formal $f$, the following statements hold.
\begin{enumerate}
\item[(i)]
For positive integers $w,w_{1},w_{2}$, we have
\[
\mathcal{F}^{f}_{0,(w)},\quad \mathcal{F}^{f}_{0,(w_{1},w_{2})}\in\bbQ[f,\Delta].
\]
\item[(ii)]
For $n>0$ and $2g-2+n>0$, we have
\[
\mathcal{F}^{f}_{g,\boldsymbol{\mu}} \in \frac{1}{\Delta^{5g-5+2n}} \bbQ[f,\Delta].
\]
\end{enumerate}}
\end{proposition}

\begin{proof}
{We first prove (i).}
For integral $f$, Proposition~\ref{prop:disk_mirror_Xp} and \eqref{equ:B_disk_lagrange_coefficients} give
\[
\mathcal{F}^{f}_{0,(w)}(\mathbf{q}) = -\frac{1}{w^{2}(1-\mathbf{q})^{(p+1)w}} \sum_{j=0}^{w} C_{w,j}^{f}\mathbf{q}^{j}.
\]
For the annulus, Lagrange inversion gives
\[
[\widehat{X}^{m}]\rho(\widehat{X})
=\frac{1}{m}[z^{m-1}]
(1-z)^{(f+1)m}(1-\mathbf{q}z)^{(1-pf)m}
\in\bbQ[p,f][\mathbf{q}]_{\leq m-1}
\]
for $m\geq 1$.
By \eqref{equ:B_disk_annulus_definitions}, the annulus
coefficient is
\[
\begin{aligned}
B_{0,(w_1,w_2)}(p,f,\mathbf{q})
&\coloneqq
[\widehat{X}_1^{w_1}\widehat{X}_2^{w_2}]
\check{\mathcal{F}}_{0,2}^{f}\\
&=
[\widehat{X}_1^{w_1}\widehat{X}_2^{w_2}]
\log\frac{\rho(\widehat{X}_1)-\rho(\widehat{X}_2)}
{\widehat{X}_1-\widehat{X}_2}
\in\bbQ[p,f][\mathbf{q}]_{\leq w_1+w_2}.
\end{aligned}
\]
For nonexceptional integral $f$, applying Proposition~\ref{prop:annulus_mirror_Xp} and the open mirror map, we obtain
\[
(1-\mathbf{q})^{(p+1)(w_{1}+w_{2})}
\mathcal{F}^{f}_{0,(w_{1},w_{2})}(\mathbf{q})
=(-1)^{w_{1}+w_{2}+1}B_{0,(w_{1},w_{2})}(p,f,\mathbf{q}).
\]
By \eqref{equ:open_vertex_gluing} and the explicit formulas above, the coefficients on both sides of the disk and
annulus identities are polynomial in $f$.
Since these identities hold for infinitely many integral framings, they extend to formal $f$.
Substituting $\mathbf{q}=(1-\Delta)/(p^{2}-\Delta)$ shows that both coefficients belong to $\bbQ[f,\Delta]$.

{We now prove (ii) using the cut-and-join equation.}
{
Recall from~\eqref{equ:connected_generating_function} that $F^{f}(\hbar,\mathbf{Q};\mathbf{p})=\log Z^{f}(\hbar,\mathbf{Q};\mathbf{p})$.
}
Since Schur functions are eigenfunctions of the cut-and-join operator (cf. \cite[p.~534]{LLLZ09}), differentiating \eqref{equ:open_vertex_gluing} gives
\begin{equation}
\frac{\partial F^{f}}{\partial f}
=\frac{\hbar}{2}\sum_{i,j\geq 1}\Bigg[
(i+j)p_{i}p_{j}
\frac{\partial F^{f}}{\partial p_{i+j}}
+ij p_{i+j}\left(
\frac{\partial^{2}F^{f}}
{\partial p_{i}\partial p_{j}}
+\frac{\partial F^{f}}{\partial p_{i}}
\frac{\partial F^{f}}{\partial p_{j}}
\right)\Bigg].
\end{equation}
Taking coefficients, we proceed by induction on $2g-2+n$, and then on $|\boldsymbol{\mu}|$.
Every coefficient $\mathcal{F}^f_{g',\boldsymbol{\nu}}$ on the
right-hand side with $2g'-2+\ell(\boldsymbol{\nu})>0$
satisfies $2g'-2+\ell(\boldsymbol{\nu})<2g-2+n$,
unless it is multiplied by a disk coefficient.
In that case, $(g',\ell(\boldsymbol{\nu}))=(g,n)$,
but $|\boldsymbol{\nu}|<|\boldsymbol{\mu}|$.
By induction, the two linear terms have pole orders at $\Delta=0$ at most $5g-7+2n$ and $5g-8+2n$, respectively, provided $2g-2+n>1$.
Products of two coefficients with $2g_{i}-2+n_{i}>0$ for $i=1,2$ have pole order at most $5g-8+2n$, since $g_{1}+g_{2}=g$ and $n_{1}+n_{2}=n+1$.
{By (i),} disk and annulus coefficients are polynomial in $\Delta$, so all remaining terms have pole order at most $5g-5+2n$. Thus
\[
\partial_{f}\mathcal{F}^{f}_{g,\boldsymbol{\mu}} \in\Delta^{-(5g-5+2n)}\bbQ[f,\Delta].
\]
Integrating from $f_{\star}$ and using {Proposition~\ref{prop:open_fstar_polynomiality}} {proves (ii).}
\end{proof}

{To study the polynomiality in $\mathbf{q}$, we first consider the B-model coefficients}
\[
B_{g,\boldsymbol{\mu}}(p,f,\mathbf{q})
\coloneqq
[\widehat{X}_1^{\mu_1}\cdots\widehat{X}_n^{\mu_n}]
\check{\mathcal{F}}_{g,n}^{f}.
\]

\begin{proposition}\label{prop:open_B_properties}
{Let $n>0$, $2g-2+n>0$, and $\boldsymbol{\mu}\in\bbZ_{>0}^{n}$. The B-model coefficients have the following properties.
\begin{enumerate}
\item[(i)] For every nonexceptional $f\neq f_{\star}$, $B_{g,\boldsymbol{\mu}}(p,f,\mathbf{q})$ is regular at $\mathbf{q}=1$.
\item[(ii)] For every nonexceptional integral framing $f$, we have
\[
B_{g,\boldsymbol{\mu}}(p,f,\mathbf{q})
=O(\mathbf{q}^{|\boldsymbol{\mu}|})
\qquad\text{as }\mathbf{q}\to\infty.
\]
\end{enumerate}}
\end{proposition}

\begin{proof}
{To prove (i), we use the graph sum for $\check{\mathcal{F}}_{g,n}^{f}$ obtained from~\eqref{equ:TR_integrated_open_leaves} and first examine the vectors $W_{w}^{\rB}$ entering the leg contributions.}
Equation \eqref{equ:B_disk_lagrange_coefficients} gives
\[
\langle\mathbf{1},W_{w}^{\rB}\rangle_{f}
=w^{2}[\widehat{X}^{w}]\check{\mathcal{F}}_{0,1}^{f}
=(-1)^{w+1}\sum_{j=0}^{w} C_{w,j}^{f}\mathbf{q}^{j}.
\]
For the pairing with $H$, we have
\begin{equation}\label{equ:outer_H_pairing_q1}
\langle H,W_{w}^{\rB}\rangle_{f}
=\left(K_{f}+\frac{D}{w}\right)
\langle\mathbf{1},W_{w}^{\rB}\rangle_{f}.
\end{equation}
Since $K_f$ and the coefficient of
$\partial/\partial\mathbf{q}$ in $D$ are regular at $\mathbf{q}=1$, both pairings are regular there.
The pairing is independent of $\mathbf{q}$ and nondegenerate, so $W_w^{\rB}$ is regular at $\mathbf{q}=1$.
{By Propositions~\ref{prop:R_properties}(i) and \ref{prop:identification_R}, each coefficient of $\check{R}^f$ is regular at $\mathbf{q}=1$.
Together with the formulas for the normalized canonical basis, this shows that every factor in the B-model graph sum, and hence $B_{g,\boldsymbol{\mu}}$, is regular there.}

{For (ii), we estimate the same graph sum as $\mathbf{q}\to\infty$, with $f$ a nonexceptional integral framing. For $w\geq1$, the expression above for $\langle\mathbf{1},W_w^{\rB}\rangle_f$ is polynomial in $\mathbf{q}$ of degree at most $w$, and hence is $O(\mathbf{q}^{w})$.}
In the coordinate $\xi=\mathbf{q}^{-1}$, we have
\[
D=-\frac{\xi-1}{\xi-p^{2}}\,
\xi\frac{\partial}{\partial\xi},
\qquad K_{f}=\frac{f}{2}+\frac{p+1}{\xi-p^{2}}.
\]
The coefficients in these expressions are regular at $\xi=0$,
and $\xi\partial_{\xi}$ preserves $\xi^{-w}\mathbb{C}[\![\xi]\!]$.
Equation~\eqref{equ:outer_H_pairing_q1} therefore implies $\langle H,W_{w}^{\rB}\rangle_{f}=O(\mathbf{q}^{w})$.
The pairing is independent of $\mathbf{q}$ and nondegenerate for nonexceptional $f$, so these two pairings imply $W_{w}^{\rB}=O(\mathbf{q}^{w})$.

Since the change from the flat basis to the normalized canonical basis is $O(1)$, the components $[\widehat{X}^{w}]\zeta_{0}^{\bar{\beta}}$ of $W_{w}^{\rB}$ are $O(\mathbf{q}^{w})$.
Using \eqref{equ:zeta_winding_coefficients},
we obtain $[\widehat{X}^{w}]\mathcal{U}_{\rB}^{f}=O(\mathbf{q}^{w})$ coefficientwise in $\fu$.
{By Propositions~\ref{prop:R_properties}(i) and \ref{prop:identification_R}, each coefficient of $\check{R}^{f}$, and hence of $(\check{R}^{f})^{-1}$, is $O(1)$,}
so \eqref{equ:B_model_open_leaf} gives
\[
[\widehat{X}^{w}]\mathcal{L}_{\rB}^{f}
=O(\mathbf{q}^{w})
\qquad\text{as }\mathbf{q}\to\infty,
\]
coefficientwise in $\fu$.
{The remaining vertex, edge, and translation factors in the B-model graph sum are coefficientwise $O(1)$ as $\mathbf{q}\to\infty$, by the same bound on $\check{R}^{f}$ and the formulas for the normalized canonical basis and pairing.}
{Thus each graph contributes $O(\mathbf{q}^{|\boldsymbol{\mu}|})$, and hence}
\[
B_{g,\boldsymbol{\mu}}(p,f,\mathbf{q})
=O(\mathbf{q}^{|\boldsymbol{\mu}|})
\qquad\text{as }\mathbf{q}\to\infty.
\]
\end{proof}

{We now turn to the A-model and also regard $p$ as a formal parameter in the topological vertex formula~\eqref{equ:open_vertex_gluing}. For $n>0$, $2g-2+n>0$, and $\boldsymbol{\mu}\in\bbZ_{>0}^{n}$, we set}
\[
\mathcal{P}_{g,\boldsymbol{\mu}}(p,f;\mathbf{q})
\coloneqq
(1-p^{2}\mathbf{q})^{5g-5+2n}
(1-\mathbf{q})^{(p+1)|\boldsymbol{\mu}|}
{\mathcal{F}^{f}_{g,\boldsymbol{\mu}}(\mathbf{q})}.
\]
{Each coefficient of $\mathcal{P}_{g,\boldsymbol{\mu}}$ is polynomial in $p,f$, as follows directly from the topological vertex formula.}
We write $Z^{f}=1+U$, so that $\log Z^{f}=\sum_{r\geq 1}(-1)^{r+1}U^{r}/r$.
Since $U$ has no constant term in $\mathbf{Q}$ and $\mathbf{p}$, only finitely many terms in the logarithmic expansion contribute to the coefficient of $\mathbf{Q}^{d}p_{\boldsymbol{\mu}}$.
For each product contributing to this coefficient, the dependence on $p,f$ is through the factor
\(
\exp(
\frac{\hbar}{2}
(f\sum_{a}\kappa_{\rho_{a}}+p\sum_{a}\kappa_{\nu_{a}})
).
\)
By \cite[Equation~(2-4)]{LLLZ09} and \eqref{equ:one_two_partition_W}, the remaining factor belongs to $\bbQ(\!(\hbar)\!)$, so only finitely many terms of this exponential contribute to the coefficient of $\hbar^{2g-2+n}$. The coefficient extraction in \eqref{equ:open_potential_recovery} therefore gives {$[\mathbf{Q}^{d}X_1^{\mu_1}\cdots X_n^{\mu_n}]\mathcal{F}^{f}_{g,n}(\mathbf{Q};X_1,\ldots,X_n)\in\bbQ[p,f]$}.
Under the closed mirror map, we have
\[
[\mathbf{q}^{m}]\mathbf{Q}^{d}
=(-1)^{m-d}\binom{d(p^{2}-1)}{m-d}\in\bbQ[p],
\qquad m\geq d\geq 0.
\]
Thus {$[\mathbf{q}^{m}]\mathcal{F}^{f}_{g,\boldsymbol{\mu}}(\mathbf{q})\in\bbQ[p,f]$}.
{The two factors in the definition of $\mathcal{P}_{g,\boldsymbol{\mu}}$ also have coefficients in $\bbQ[p]$, so
\begin{equation}\label{equ:open_coefficientwise_polynomiality}
\mathcal{P}_{g,\boldsymbol{\mu}}(p,f;\mathbf{q})\in\bbQ[p,f][\![\mathbf{q}]\!].
\end{equation}}

\begin{theorem}
    \label{thm:open_formal_p_polynomiality}
    Let $n>0$, $2g-2+n>0$, and $\boldsymbol{\mu}\in\bbZ_{>0}^{n}$.  The open potentials admit the expression
    \begin{equation}
    (1-\mathbf{q})^{(p+1)|\boldsymbol{\mu}|}
    \mathcal{F}^{f}_{g,\boldsymbol{\mu}}(\mathbf{q})
    =
    \frac{\mathcal{P}_{g,\boldsymbol{\mu}}(p,f;\mathbf{q})}
    {(1-p^{2}\mathbf{q})^{5g-5+2n}}
    \end{equation}
    for all integers $p\geq 2$ and formal $f$, where
    \[
    \mathcal{P}_{g,\boldsymbol{\mu}}(p,f;\mathbf{q})
    \in
    \bbQ[p,f,\mathbf{q}],
    \qquad
    \deg_{\mathbf{q}}\mathcal{P}_{g,\boldsymbol{\mu}}
    \leq 5g-5+2n+|\boldsymbol{\mu}|.
    \]
\end{theorem}

\begin{proof}
{For nonexceptional integral $f$, Theorem~\ref{thm:open_mirror_Xp}, the open mirror map, and the definition of $\mathcal{P}_{g,\boldsymbol{\mu}}$ give
\begin{equation}\label{equ:N_B_relation}
\mathcal{P}_{g,\boldsymbol{\mu}}(p,f;\mathbf{q})
=(-1)^{g-1+n+|\boldsymbol{\mu}|}
(1-p^{2}\mathbf{q})^{5g-5+2n}
B_{g,\boldsymbol{\mu}}(p,f,\mathbf{q}).
\end{equation}
Both sides are rational in $f$ and $\mathbf{q}$ by topological recursion and Proposition~\ref{prop:open_framing_filtration}(ii). Since the identity holds for infinitely many integral framings, it extends to formal $f$.}

{Since $\Delta=(1-p^{2}\mathbf{q})/(1-\mathbf{q})$, Proposition~\ref{prop:open_framing_filtration}(ii) shows that $\mathcal{P}_{g,\boldsymbol{\mu}}$ has no finite poles except possibly at $\mathbf{q}=1$. By Proposition~\ref{prop:open_B_properties}(i) and~\eqref{equ:N_B_relation}, it is also regular there for nonexceptional $f\neq f_{\star}$. Polynomial dependence on $f$ then gives
\begin{equation}\label{equ:open_fixed_p_polynomiality}
\mathcal{P}_{g,\boldsymbol{\mu}}(p,f;\mathbf{q})
\in
\bbQ[f,\mathbf{q}].
\end{equation}

Finally, for nonexceptional integral $f$, Proposition~\ref{prop:open_B_properties}(ii) and~\eqref{equ:N_B_relation} give}
{\[
    \deg_{\mathbf{q}}\mathcal{P}_{g,\boldsymbol{\mu}}(p,f;\mathbf{q})
    \leq 5g-5+2n+|\boldsymbol{\mu}|.
    \]
For $m>5g-5+2n+|\boldsymbol{\mu}|$, set $a_m(p,f)\coloneqq[\mathbf{q}^{m}]\mathcal{P}_{g,\boldsymbol{\mu}}(p,f;\mathbf{q})$. By~\eqref{equ:open_coefficientwise_polynomiality}, we have $a_m(p,f)\in\bbQ[p,f]$.}
For each integer $p_0\geq 2$, the degree estimate implies $a_m(p_0,f_0)=0$ for infinitely many integral $f_0$.
Hence $a_m(p_0,f)=0$.
Since this holds for infinitely many $p_0$, we obtain $a_m(p,f)=0$, completing the proof.
\end{proof}

\section{Double-scaling limits and pure gravity}
\label{sec:double_scaling}

In this section, we determine the double-scaling limits of the spectral curves of $X_p$, their stable topological recursion differentials, and the closed free energies for all fixed framings. 
At $f=0$, the disk and annulus limits agree with those of~\cite{Mar08},
and the closed limits yield another proof of the pure-gravity limit conjectured in~\cite{CGMPS07} and subsequently proved by Mari\~no~\cite{Mar09}.

\subsection{Double-scaling limits of the spectral curve}

We consider the spectral curve~\eqref{equ:SpecCurv} with $\mathbf{q}$ replaced by
\begin{equation}\label{equ:double_scaling_parameters}
\mathbf{q}_{\varepsilon}
\coloneqq\frac{1-\varepsilon^{2}\tau}{p^{2}-\varepsilon^{2}\tau},
\end{equation}
where $\tau\in\bbC^{*}$ is fixed.
Then $\mathbf{q}_{\varepsilon}\to p^{-2}$ as $\varepsilon\to 0$, and
\[
\Delta(\mathbf{q}_{\varepsilon})
=\frac{1-p^{2}\mathbf{q}_{\varepsilon}}{1-\mathbf{q}_{\varepsilon}}
=\varepsilon^{2}\tau.
\]

At $\mathbf{q}=p^{-2}$, both $\rd x_{f}$ and $\rd y$ vanish at $z=-p$.  We therefore introduce the rescaled coordinate $u=(z+p)/\varepsilon$, with
\[
\iota_{\varepsilon}(u)\coloneqq -p+\varepsilon u.
\]
For $f\neq f_{\star}$, we set
\begin{equation}\label{equ:critical_y_shear}
\widetilde{y}_{f}
\coloneqq y-\frac{x_{f}}{f-f_{\star}},
\end{equation}
which removes the quadratic term in the expansion at $z=-p$ when $\mathbf{q}=p^{-2}$.
The Eynard--Orantin differentials for $2g-2+n>0$ and the free energies are unchanged by this replacement.

\begin{proposition}\label{prop:critical_spectral_curve}
For $f\neq f_{\star}$, the functions
\[
\begin{aligned}
x_{f,\varepsilon}(u)
&\coloneqq\varepsilon^{-2}
\left(
x_{f}(\iota_{\varepsilon}(u);\mathbf{q}_{\varepsilon})
-x_{f}(-p;\mathbf{q}_{\varepsilon})
\right),\\
y_{f,\varepsilon}(u)
&\coloneqq\varepsilon^{-3}
\left(
\widetilde{y}_{f}(\iota_{\varepsilon}(u);\mathbf{q}_{\varepsilon})
-\widetilde{y}_{f}(-p;\mathbf{q}_{\varepsilon})
\right)
\end{aligned}
\]
converge locally uniformly on $\bbC$, as $\varepsilon\to 0$, to
\begin{equation}\label{equ:generic_double_scaled_curve}
x_{f}^{\mathrm{ds}}(u)
=\frac{(p-1)(f-f_{\star})}{2p(p+1)^{2}}u^{2},
\qquad
y_{f}^{\mathrm{ds}}(u)
=\frac{1}{p^{2}(p+1)^{2}(f-f_{\star})}
\left(
\frac{u^{3}}{3}-(p+1)^{2}\tau u
\right).
\end{equation}
At $f=f_{\star}$, the functions
\[
\begin{aligned}
x_{f_{\star},\varepsilon}(u)
&\coloneqq\varepsilon^{-3}
\left(
x_{f_{\star}}(\iota_{\varepsilon}(u);\mathbf{q}_{\varepsilon})
-x_{f_{\star}}(-p;\mathbf{q}_{\varepsilon})
\right),\\
y_{f_{\star},\varepsilon}(u)
&\coloneqq\varepsilon^{-2}
\left(
y(\iota_{\varepsilon}(u);\mathbf{q}_{\varepsilon})
-y(-p;\mathbf{q}_{\varepsilon})
\right)
\end{aligned}
\]
converge in the same sense to
\begin{equation}\label{equ:special_double_scaled_curve}
x_{f_{\star}}^{\mathrm{ds}}(u)
=-\frac{1}{p^{2}(p+1)^{2}}
\left(
\frac{u^{3}}{3}-(p+1)^{2}\tau u
\right),
\qquad
y_{f_{\star}}^{\mathrm{ds}}(u)
=\frac{p-1}{2p(p+1)^{2}}u^{2}.
\end{equation}
In both cases, the Bergman kernel satisfies
\[
(\iota_{\varepsilon}\times\iota_{\varepsilon})^{*}B=\frac{\rd u_{1}\rd u_{2}}{(u_{1}-u_{2})^{2}}.
\]
\end{proposition}

\begin{proof}
We first reduce the case $f\neq f_{\star}$ to $f=f_{\star}$.
By \eqref{equ:SpecCurv}, we have
\[ x_f=x_{f_{\star}}+(f-f_{\star})y. \]
For $f\neq f_{\star}$, it follows that
\[
x_{f,\varepsilon}
=\varepsilon x_{f_{\star},\varepsilon}
+(f-f_{\star})y_{f_{\star},\varepsilon},
\qquad
y_{f,\varepsilon}
=-\frac{x_{f_{\star},\varepsilon}}{f-f_{\star}}.
\]
Since $x_{f_{\star},\varepsilon}(0)=y_{f_{\star},\varepsilon}(0)=0$,
it suffices to prove local uniform convergence of their derivatives.
Substituting \eqref{equ:double_scaling_parameters} and $z=-p+\varepsilon u$ into the expressions for $x_{f_{\star}}'$ and $y'$ from \eqref{equ:xf_prime_Nf} and \eqref{equ:SpecCurv}, we obtain
\[
\begin{aligned}
\lim_{\varepsilon\to 0}
\varepsilon^{-2}x_{f_{\star}}'
(-p+\varepsilon u;\mathbf{q}_{\varepsilon})
&=-\frac{u^{2}-(p+1)^{2}\tau}{p^{2}(p+1)^{2}},\\
\lim_{\varepsilon\to 0}
\varepsilon^{-1}y'
(-p+\varepsilon u;\mathbf{q}_{\varepsilon})
&=\frac{p-1}{p(p+1)^{2}}u.
\end{aligned}
\]
After cancellation of the powers of $\varepsilon$, these derivatives extend holomorphically across $\varepsilon=0$, locally in $u$.
Integration from $u=0$ then gives the claimed convergence.
The formulas for general $f$ follow from the identities above, and the formula for $B$ follows by direct substitution.
\end{proof}

To analyze the residues in topological recursion, we next determine the critical points of $x_{f}$ and their behavior as $\varepsilon \to 0$.

\begin{lemma}
\label{lem:critical_branch_structure}
For all sufficiently small $\varepsilon\neq 0$, the critical points of $x_f$ at $\mathbf{q}=\mathbf{q}_{\varepsilon}$ are simple, and $\rd y$ is nonzero at each of them. Their behavior as $\varepsilon\to 0$ is as follows.
\begin{enumerate}
\item[(i)] If $f\neq f_{\star}$, exactly one critical point tends to $-p$, and its rescaled coordinate $u=(z+p)/\varepsilon$ tends to $0$. If a second critical point exists, it converges to a simple critical point of $x_f$ at $\mathbf{q}=p^{-2}$, distinct from $-p$.
\item[(ii)] If $f=f_{\star}$, there are exactly two critical points, whose $u$-coordinates tend to ${\pm(p+1)\sqrt{\tau}}$.
\end{enumerate}
\end{lemma}

\begin{proof}
Suppose first that $f\neq f_{\star}$. At $\mathbf{q}=p^{-2}$, we have
\[
N_{f}(z)=\frac{(z+p)(\gamma_{f}z+p)}{p^{2}}.
\]
For $f=-1$ and $f=p^{-1}$, the factors $1-z$ and $1-\mathbf{q}z$, respectively, cancel with the denominator in \eqref{equ:xf_prime_Nf}. When $\gamma_f=0$, the polynomial $N_f$ is linear.
Thus all zeros of $\rd x_f$ at $\mathbf{q}=p^{-2}$ are simple and extend holomorphically in $\mathbf{q}$.
Since $\mathbf{q}_{\varepsilon}-p^{-2}=O(\varepsilon^2)$, the zero tending to $-p$ satisfies $z=-p+O(\varepsilon^2)$ and hence $u=O(\varepsilon)$.
Any second critical point tends to $-p/\gamma_f=-p/(1+(p-1)(f-f_{\star}))\neq -p$, proving (i).

For $f=f_{\star}$, the proof of Proposition~\ref{prop:critical_spectral_curve} shows that $\partial_{u}x_{f_{\star},\varepsilon}$ extends holomorphically to $\varepsilon=0$, with simple zeros $u=\pm(p+1)\sqrt{\tau}$.
These simple zeros extend to $u=\pm(p+1)\sqrt{\tau}+O(\varepsilon)$, proving (ii).

Finally, the unique finite zero of $y'$ is $z_{y}=-(1-p\mathbf{q})/(\mathbf{q}(p-1))$, and
\[
N_{f}(z_{y})
=-\frac{(\mathbf{q}-1)(p^{2}\mathbf{q}-1)}
{\mathbf{q}(p-1)^{2}}
\neq 0
\]
at $\mathbf{q}=\mathbf{q}_{\varepsilon}$ for sufficiently small $\varepsilon\neq 0$. Thus $\rd y$ is nonzero at every zero of $\rd x_{f}$.
\end{proof}

To compare these limits with minimal models, we consider the spectral curves
\begin{equation}\label{equ:unit_minimal_model_curves}
\begin{alignedat}{3}
\mathsf{M}_{3,2}:\quad&&
x(v)&=\frac{1}{2}(v^{2}-2),
\qquad&
y(v)&= -\frac{1}{3}(v^{3}-3v),\\
\mathsf{M}_{2,3}:\quad&&
x(v)&=\frac{1}{3}(v^{3}-3v),
\qquad&
y(v)&=\frac{1}{2}(v^{2}-2).
\end{alignedat}
\end{equation}
Both curves are equipped with the Bergman kernel $\frac{\rd v_{1}\rd v_{2}}{(v_{1}-v_{2})^{2}}$.

Under the change of coordinates $u=(p+1)\sqrt{\tau}\,v$,
the limiting curves \eqref{equ:generic_double_scaled_curve} and \eqref{equ:special_double_scaled_curve} are obtained from $\mathsf{M}_{3,2}$ and $\mathsf{M}_{2,3}$, respectively, by rescaling $x,y$ and adding constants.

\subsection{Double-scaling limits in topological recursion}
\label{subsec:double_scaling_TR}

Let $\omega_{g,n}^{f}(\mathbf{q})$ denote the Eynard--Orantin differentials of the spectral curve \eqref{equ:SpecCurv}.
We write $\omega_{g,n}^{f,\tau,\mathrm{ds}}$ and $F_{g}^{f,\tau,\mathrm{ds}}$ for the differential and free energy of \eqref{equ:generic_double_scaled_curve} when $f\neq f_{\star}$, and for those of \eqref{equ:special_double_scaled_curve} when $f=f_{\star}$.

Note that under $x\mapsto ax$ and $y\mapsto by$, with $a,b\in\bbC^{*}$ and $B$ unchanged, topological recursion and the residue formula for the free energy give
\[
\omega_{g,n}\longmapsto(ab)^{2-2g-n}\omega_{g,n},
\qquad
F_{g}\longmapsto(ab)^{2-2g}F_{g},
\]
for $2g-2+n>0$ and $g\geq 2$, respectively.
Moreover, adding to $y$ a function of $x$ that is holomorphic near the critical values of $x$ leaves $\omega_{g,n}$ for $2g-2+n>0$ and $F_g$ for $g\geq 2$ unchanged.

\begin{lemma}\label{lem:common_ds_free_energy}
For $g\geq 2$, the free energy $F_{g}^{f,\tau,\mathrm{ds}}$ is independent of $f$ and satisfies
\[
F_{g}^{f,\tau,\mathrm{ds}}
=\frac{p^{6g-6}}{(p^{2}-1)^{2g-2}}
\tau^{5-5g}F_{g}(\mathsf{M}_{3,2}),
\]
where $F_{g}(\mathsf{M}_{3,2})$ denotes the genus-$g$ free energy of the
spectral curve $\mathsf{M}_{3,2}$ in \eqref{equ:unit_minimal_model_curves}.
\end{lemma}

\begin{proof}
For the rescalings described above, the product of the factors multiplying $x$ and $y$ is $-(p^{2}-1)\tau^{5/2}/p^{3}$ in both cases.
Exchanging $x$ and $y$ and replacing $v$ by $-v$ identifies $\mathsf{M}_{3,2}$ with $\mathsf{M}_{2,3}$.
By \cite[Equation (4-1)]{EO13}, we have $F_g(\mathsf{M}_{2,3})=F_g(\mathsf{M}_{3,2})$.
The stated formula follows from the rescaling formula above.
\end{proof}

\begin{lemma}\label{lem:critical_regular_recursion_weights}
For $n>0$ and $2g-2+n>0$, using \eqref{equ:EO_recursion}, we obtain an iterated residue expansion of
\[ \varepsilon^{5(2g-2+n)}(\iota_{\varepsilon}^{\times n})^{*}\omega_{g,n}^{f}(\mathbf{q}_{\varepsilon}). \]
As $\varepsilon\to 0$, the terms with all residue points tending to $-p$
converge to the corresponding terms for the limiting curve,
while all other terms tend to zero.
Both limits are locally uniform away from critical points,
poles, and diagonals of the limiting curve.
\end{lemma}

\begin{proof}
The iterated residue expansion of $\omega_{g,n}^{f}(\mathbf{q}_{\varepsilon})$ is a finite sum involving recursion kernels and copies of the Bergman kernel $B$.
We denote one such term by $I_{\varepsilon}$, with a critical point specified at each residue.
Each term contains $2g-2+n$ recursion kernels.

By Proposition~\ref{prop:critical_spectral_curve} and Lemma~\ref{lem:critical_branch_structure}, we may choose local Airy coordinates for the curves defined by $x_{f,\varepsilon}$ and $y_{f,\varepsilon}$ so that these coordinates converge to those of the limiting curve.
We compute the iterated residues on fixed nested circles in these coordinates, using unscaled coordinates at the other critical points and avoiding the poles of $B$.
The estimates below are uniform on these contours and locally uniform away from critical points, poles, and diagonals.

For $f\neq f_{\star}$, replacing $y$ by $\widetilde y_f$ leaves the recursion kernel unchanged.
In both cases, Proposition~\ref{prop:critical_spectral_curve} shows that,
near each critical point tending to $-p$,
the denominator of the recursion kernel, expressed in $u$,
is $\varepsilon^5$ times the corresponding denominator for the rescaled curve.
On the contours around the other critical points, the denominators converge uniformly to nonvanishing functions.
The Bergman kernel is unchanged when both variables are expressed in $u$ and is $O(1)$ when both lie near the other critical points.
When only one variable is rescaled, with $z'$ near a critical point whose limit is distinct from $-p$, it becomes
\[
\frac{\varepsilon\,\rd u\,\rd z'}
{(-p+\varepsilon u-z')^{2}}
=O(\varepsilon).
\]
Integration in the numerator of a recursion kernel preserves this estimate.

For each term, we let $k$ denote the number of recursion kernels whose residues are taken at critical points tending to $-p$.
We let $\ell$ denote the number of factors of $B$ with one variable near these critical points and the other near a remaining critical point, including the factors integrated in the numerators of the recursion kernels.
The preceding estimates give
\begin{equation}\label{equ:critical_regular_term_weight}
I_{\varepsilon}=O\left(\varepsilon^{-5k+\ell}\right).
\end{equation}

We now take $z_i=\iota_{\varepsilon}(u_i)$ for every $i$.
If $k=2g-2+n$, uniform convergence of the rescaled integrands
on the fixed contours implies that
\[
\varepsilon^{5(2g-2+n)}
(\iota_{\varepsilon}^{\times n})^{*}I_{\varepsilon}
\]
converges to the corresponding iterated residue for the limiting curve.
If $k<2g-2+n$, the rescaled term is
$O(\varepsilon^{5(2g-2+n-k)+\ell})\to 0$.
\end{proof}

We now apply these estimates to obtain the double-scaling limits of the multidifferentials and free energies.

\begin{theorem}
\label{thm:double_scaling_TR}
For $n>0$ and $2g-2+n>0$, we have
\begin{equation}\label{equ:stable_double_scaling_limit}
\lim_{\varepsilon\to 0}
\varepsilon^{5(2g-2+n)}
(\iota_{\varepsilon}^{\times n})^{*}
\omega_{g,n}^{f}(\mathbf{q}_{\varepsilon}) =
\omega_{g,n}^{f,\tau,\mathrm{ds}}.
\end{equation}
For $g\geq 2$, we have
\begin{equation}\label{equ:closed_double_scaling_limit}
\lim_{\varepsilon\to 0}
\varepsilon^{10g-10}
\check{\mathcal{F}}_{g}(\mathbf{q}_{\varepsilon})=F_{g}^{f,\tau,\mathrm{ds}}.
\end{equation}
\end{theorem}

\begin{proof}
Summing the termwise limits in Lemma~\ref{lem:critical_regular_recursion_weights} gives \eqref{equ:stable_double_scaling_limit}.

To prove \eqref{equ:closed_double_scaling_limit}, we choose a nonexceptional framing $f_{0}\neq f_{\star}$.
By \eqref{equ:TR_CohFT_closed}, we have $\check{\mathcal{F}}_{g}=F_{g}(\calC_{f_{0}})$, where $F_{g}(\calC_{f_{0}})$ denotes the genus-$g$ free energy of $\calC_{f_{0}}$.
By the invariance under adding a function of $x$ to $y$,  we may replace $y$ by $\widetilde{y}_{f_{0}}$ as in \eqref{equ:critical_y_shear}.
At $\mathbf{q}=\mathbf{q}_{\varepsilon}$, the definition of free energy gives
\[
F_{g}(\calC_{f_{0}})
=\frac{1}{2-2g}\sum_{a}
\operatorname{Res}_{z=a}
\Phi_{\varepsilon,a}(z)\omega_{g,1}^{f_{0}}(z),
\qquad
\Phi_{\varepsilon,a}(z)
\coloneqq\int_{a}^{z}
\bigl(\widetilde{y}_{f_{0}}(w)-\widetilde{y}_{f_{0}}(a)\bigr)
\rd x_{f_{0}}(w).
\]
At the critical point $a$ tending to $-p$, we have
\[
\Phi_{\varepsilon,a}(\iota_{\varepsilon}(u))
=\varepsilon^{5}\left(
\int_{0}^{u} y_{f_0}^{\mathrm{ds}}(v)\,\rd x_{f_0}^{\mathrm{ds}}(v)+O(\varepsilon)
\right).
\]
Applying \eqref{equ:stable_double_scaling_limit} with $n=1$ shows that this contribution, multiplied by $\varepsilon^{10g-10}$, converges to $F_{g}^{f_{0},\tau,\mathrm{ds}}$.

At the other critical point $a$, we have $\Phi_{\varepsilon,a}=O(1)$.
The same estimates for the recursion kernels and $B$ apply near $z=a$, with $k,\ell$ defined as in the proof of Lemma~\ref{lem:critical_regular_recursion_weights}.
Only terms in the expansion of $\omega_{g,1}^{f_0}(z)$ whose outermost residue is taken at $a$ contribute,
since all other terms are holomorphic at $z=a$.
Each such term contains $2g-1$ recursion kernels, so $k\leq 2g-2$.
By \eqref{equ:critical_regular_term_weight}, its contribution after
multiplication by $\varepsilon^{10g-10}$ is
$O(\varepsilon^{5(2g-2-k)+\ell})$.
If $k\leq 2g-3$, this is $O(\varepsilon^{5})$.
If $k=2g-2$, all residues except the outermost one are taken at critical points tending to $-p$.
The numerator of a kernel in the next recursion step therefore contains an integral of $B$ with one variable near $a$ and the other near $-p$. Hence $\ell\geq 1$, and the contribution is $O(\varepsilon)$.
Thus the rescaled free energy converges to $F_g^{f_0,\tau,\mathrm{ds}}$, which equals $F_g^{f,\tau,\mathrm{ds}}$ by
Lemma~\ref{lem:common_ds_free_energy}.
This proves \eqref{equ:closed_double_scaling_limit}.
\end{proof}

We next consider the disk and annulus differentials.
At $\mathbf{q}=\mathbf{q}_{\varepsilon}$, we define
\[
\widetilde{\omega}_{0,1}^{f}
\coloneqq
\begin{cases}
\bigl(\widetilde{y}_{f}(z)-\widetilde{y}_{f}(-p)\bigr)\rd x_{f}(z),
& f\neq f_{\star},\\
\bigl(y(z)-y(-p)\bigr)\rd x_{f_{\star}}(z),
& f=f_{\star}.
\end{cases}
\]
Then Proposition~\ref{prop:critical_spectral_curve} gives
\begin{equation}\label{equ:disk_double_scaling_limit}
\varepsilon^{-5}\iota_{\varepsilon}^{*}
\widetilde{\omega}_{0,1}^{f}
\longrightarrow
y_{f}^{\mathrm{ds}}\,\rd x_{f}^{\mathrm{ds}}.
\end{equation}

For the annulus differential in \eqref{equ:B_disk_annulus_definitions},
Proposition~\ref{prop:critical_spectral_curve}, together with $\widehat{X}=\ee^{x_f}$, gives
\begin{equation}\label{equ:annulus_double_scaling_limit}
(\iota_{\varepsilon}\times\iota_{\varepsilon})^{*}
\left(
B-\frac{\rd\widehat{X}_{1}\rd\widehat{X}_{2}}
{(\widehat{X}_{1}-\widehat{X}_{2})^{2}}
\right)
\longrightarrow
\frac{\rd u_{1}\rd u_{2}}{(u_{1}-u_{2})^{2}}
-\frac{\rd x_{f}^{\mathrm{ds}}(u_{1})\,\rd x_{f}^{\mathrm{ds}}(u_{2})}
{\bigl(x_{f}^{\mathrm{ds}}(u_{1})-x_{f}^{\mathrm{ds}}(u_{2})\bigr)^{2}}.
\end{equation}

\subsection{Comparison with pure gravity}

We compute the coefficients of the leading singular terms in the closed free energies and compare them with the predictions of~\cite{CGMPS07}.
For the open sector, we express the limiting multidifferentials in terms of intersection numbers and compare the disk and annulus amplitudes at zero framing with Mari\~no's calculations~\cite{Mar08}.

\subsubsection{Closed free energies}

It was conjectured in~\cite[Equations~(6.20)--(6.21)]{CGMPS07} that the double-scaled free energy agrees with that of pure gravity in the genus expansion.

\begin{corollary}\label{cor:Pg0p}
For $g\geq 2$,
\[
\calP_{g}(0,p^{2})
=
\frac{p^{6g-6}}{(p^{2}-1)^{2g-2}}
\frac{1}{(3g-3)!}
\left\langle\tau_{2}^{3g-3}\right\rangle_{g}.
\]
\end{corollary}

\begin{proof}
For $\mathsf{M}_{3,2}$ in \eqref{equ:unit_minimal_model_curves}, we have $R=\operatorname{id}$ and $T(\fu)=\fu^{2}\bar{e}$.
Equation~\eqref{equ:TR_CohFT_closed} therefore gives
\[
F_{g}(\mathsf{M}_{3,2}) = \frac{1}{(3g-3)!}\left\langle\tau_{2}^{3g-3}\right\rangle_{g}.
\]
By Theorems~\ref{thm:FgAB} and~\ref{thm:FgXp}, together with $\Delta(\mathbf{q}_{\varepsilon})=\varepsilon^{2}\tau$, we have
\[
\lim_{\varepsilon\to0}
\varepsilon^{10g-10}
\check{\mathcal{F}}_{g}(\mathbf{q}_{\varepsilon})
=\tau^{5-5g}\calP_g(0,p^2).
\]
Comparing this with \eqref{equ:closed_double_scaling_limit} and Lemma~\ref{lem:common_ds_free_energy} proves the assertion.
\end{proof}

\begin{remark}\label{rem:notation}
The threefold denoted $X_{p+1}$ in~\cite{CGMPS07,Eyn08} is our $X_p$.
The variables $w,w_c$ in~\cite{CGMPS07} are related to our notation by
\[
w=1-\mathbf{q},
\qquad
w_c=1-\frac{1}{p^2},
\qquad
\Delta=\frac{p^2(w-w_c)}{w}.
\]
We denote the free energies of~\cite{CGMPS07} by $F_g^{\mathrm{CGMPS}}$.
Their conventions differ from ours by a factor $(-1)^{g-1}$, and their free energies omit degree-zero contributions;
see~\cite[Equation~(2.10)]{CGMPS07}.
Thus, for $g\geq2$,
\[
F_g^{\mathrm{CGMPS}}(w)=(-1)^{g-1}\left(\mathcal{F}_{g}(\mathbf{Q})-N_{g,0}\right).
\]
The ansatz of~\cite[Equation~(6.2)]{CGMPS07} is
\[
F_g^{\mathrm{CGMPS}}(w)=\frac{\widetilde{\calP}_g(w,p)}{(w-w_c)^{5g-5}},
\]
where $\widetilde{\calP}_g(w,p)$ is a polynomial in $w$ of degree at most $5g-5$ satisfying $\widetilde{\calP}_g(1,p)=0$.
Comparing the coefficients of $(w-w_c)^{5-5g}$ in this
expression and Theorem~\ref{thm:FgXp} gives
\[
\widetilde{\calP}_g(w_c,p)
=\calP_g(0,p^2)
\left(\frac{1}{p^2}-\frac{1}{p^4}\right)^{5g-5}.
\]
\end{remark}

Let $g_s$ be the string coupling. Following \cite[Equation~(6.11)]{CGMPS07}, we define $\fz$ by
\[
\fz^{5/2}=g_{s}^{-2}\frac{p^{8}}{4w_{c}^{3}}(w-w_{c})^{5}.
\]
For fixed $\fz\neq0$ and $\mathbf{q}=\mathbf{q}_{\varepsilon}$, this identity determines $g_s^2$.
Then $w\to w_{c}$ and $g_{s}\to 0$ simultaneously as $\varepsilon\to 0$ with
\[
\lim_{\varepsilon\to 0}\frac{g_{s}^{2}}{\varepsilon^{10}}
=\frac{(p^{2}-1)^{2}}{4p^{6}}
\frac{\tau^{5}}{\fz^{5/2}}.
\]
Equation~\eqref{equ:closed_double_scaling_limit} and Remark~\ref{rem:notation} therefore imply
\[
\lim_{\varepsilon\to 0}
g_{s}^{2g-2}
F_{g}^{\mathrm{CGMPS}}(1-\mathbf{q}_{\varepsilon})
=c_{g}\fz^{-5(g-1)/2},\quad g\geq 2,
\]
where
\[
c_{g}=\left(\frac{p^{8}}{4w_{c}^{3}}\right)^{g-1}
\widetilde{\calP}_{g}(w_{c},p).
\]
Together with the genus-zero and genus-one terms in~\cite[Equation~(6.14)]{CGMPS07}, this yields the following genus expansion of the double-scaled free energy:
\[
-\frac{4}{15}\fz^{5/2}-\frac{1}{48}\log\fz
  +\sum_{g\geq 2}c_{g}\fz^{-5(g-1)/2}.
\]
\begin{corollary}\label{cor:dslimit}
For $g\geq 2$, the coefficients of the double-scaled free energy satisfy
\[
c_{g}=\frac{4^{1-g}}{(3g-3)!}
\left\langle\tau_{2}^{3g-3}\right\rangle_{g}.
\]
Consequently, its genus expansion agrees with that of pure gravity.
\end{corollary}

\begin{proof}
The formula follows from Corollary~\ref{cor:Pg0p} and Remark~\ref{rem:notation}.
By \cite[Equation~(6.23)]{CGMPS07}, these are the coefficients of the free energy of pure gravity for $g\geq 2$.
The genus-zero and genus-one terms also agree, proving the assertion.
\end{proof}

\subsubsection{Open string amplitudes}
\label{subsec:Marino_critical_comparison}

For $f\neq f_{\star}$, the limiting multidifferentials admit the
following expression in terms of intersection numbers.

\begin{corollary}\label{cor:open_ds_intersection_numbers}
Assume $f\neq f_{\star}$ and set $\widetilde{\iota}_{\varepsilon}(v) \coloneqq -p+\varepsilon(p+1)\sqrt{\tau}\,v$.
For $n>0$ and $2g-2+n>0$, we have
\[
\begin{aligned}
\lim_{\varepsilon\to 0}
\varepsilon^{5(2g-2+n)}
(\widetilde{\iota}_{\varepsilon}^{\times n})^{*}
\omega_{g,n}^{f}(\mathbf{q}_{\varepsilon})
&=\left(\frac{p^{3}}{p^{2}-1}\right)^{2g-2+n}\tau^{-5(2g-2+n)/2}\\
&\quad\times
\sum_{\substack{d_{1},\ldots,d_{n},k\geq 0\\
d_{1}+\cdots+d_{n}+k=3g-3+n}}
\frac{\left\langle\tau_{d_{1}}\cdots\tau_{d_{n}}\tau_{2}^{k}\right\rangle_{g}}{k!}
\prod_{i=1}^{n}\frac{(2d_{i}+1)!!\,\rd v_{i}}{v_{i}^{2d_{i}+2}}.
\end{aligned}
\]
\end{corollary}

\begin{proof}
For $\mathsf{M}_{3,2}$, the local Airy coordinate is $v$, and
\[
R=\operatorname{id},\qquad T(\fu)=\fu^{2}\bar{e},
\qquad
\rd\zeta_{d}(v)=-\frac{(2d+1)!!\,\rd v}{v^{2d+2}}.
\]
By \eqref{equ:TR_CohFT_open_leaves}, the multidifferential $\omega_{g,n}$ of $\mathsf{M}_{3,2}$ equals $(-1)^{n}$ times the sum in the statement.
Rescaling $x$ and $y$ under $u=(p+1)\sqrt{\tau}\,v$ then multiplies $\omega_{g,n}$ by
\[ (-1)^{2g-2+n}\left(\frac{p^{3}}{p^{2}-1}\right)^{2g-2+n}\tau^{-5(2g-2+n)/2}. \]
The assertion then follows from Theorem~\ref{thm:double_scaling_TR}.
\end{proof}

Mari\~no studied the outer brane at zero framing and identified the double-scaled disk and annulus amplitudes, up to normalization, with those of the FZZT brane of pure gravity~\cite[Section~3.3]{Mar08}.
The parameters $p$ and $\zeta$ in~\cite{Mar08} correspond to $p+1$ and $\mathbf{q}$ in our notation, respectively.
The parameter $\fz$ here agrees with that in~\cite[Equation~(3.18)]{Mar08}.
At $f=0$, Mari\~no's variable $\lambda$ is related to
$\widehat{X}$ by
\[
\lambda
=-\frac{(1-\mathbf{q})^{-(p+1)}}{\widehat{X}(z)}
=(1-\mathbf{q})^{-(p+1)}
\left(1+\mathbf{q}-\mathbf{q}z-z^{-1}\right).
\]
Its critical values are $(1-\mathbf{q})^{-(p+1)}(1\pm\sqrt{\mathbf{q}})^{2}$, which are exactly $x_{1},x_{2}$ in~\cite[Equation~(3.15)]{Mar08}.
At $\mathbf{q}=\mathbf{q}_{\varepsilon}$ and
$z=\widetilde{\iota}_{\varepsilon}(v)$, with $s=v^2/2-1$,
we obtain
\[
\lambda=(1-p^{-2})^{-(p+1)}
\left[
\frac{(p+1)^2}{p^2}
-\frac{2p(p+1)}{p-1}
(\mathbf{q}_{\varepsilon}-p^{-2})s
\right]
+O(\varepsilon^3).
\]
Thus $s$ agrees in the limit with the rescaled open coordinate
in~\cite[Equations~(3.39)--(3.40)]{Mar08}.

For the disk, the function $y_{\mathrm{Mar}}(\lambda)$ in~\cite[Equation~(3.31)]{Mar08} satisfies
\[
y_{\mathrm{Mar}}(\lambda)\,\rd\lambda
=-2\left(\widetilde{y}_{0}(z)
-\widetilde{y}_{0}(-\mathbf{q}^{-1/2})\right)
\rd x_{f=0}(z),
\]
where $\widetilde{y}_{0}=\widetilde{y}_{f}|_{f=0}$.
Replacing $\widetilde{y}_{0}(-\mathbf{q}_{\varepsilon}^{-1/2})$ by $\widetilde{y}_{0}(-p)$ does not affect the limit.
We choose the sign of $g_{s}$ so that
\[
\lim_{\varepsilon\to 0} \frac{g_{s}}{\varepsilon^{5}}=\frac{p^{2}-1}{2p^{3}}\tau^{5/2}\fz^{-5/4}.
\]
Equation~\eqref{equ:disk_double_scaling_limit} therefore yields, in the $s$-coordinate,
\[
-\frac{2}{g_{s}}\iota_{\varepsilon}^{*}
\widetilde{\omega}_{0,1}^{0}
\longrightarrow
-\frac{4\sqrt{2}}{3}\fz^{5/4}(2s-1)\sqrt{1+s}\,\rd s,
\]
where $\sqrt{1+s}=v/\sqrt{2}$.
This agrees with~\cite[Equation~(3.41)]{Mar08}.

For the annulus, the relation $\lambda=-(1-\mathbf{q})^{-(p+1)}/\widehat{X}$ implies
\[
\frac{\rd\lambda_{1}\rd\lambda_{2}}{(\lambda_{1}-\lambda_{2})^{2}}
=
\frac{\rd\widehat{X}_{1}\rd\widehat{X}_{2}}
{(\widehat{X}_{1}-\widehat{X}_{2})^{2}}.
\]
Expressing \eqref{equ:annulus_double_scaling_limit}
in the variables $s=v_{1}^{2}/2-1$ and $t=v_{2}^{2}/2-1$ gives
\[
\frac{\rd u_{1}\rd u_{2}}{(u_{1}-u_{2})^{2}}-
\frac{\rd x_{0}^{\mathrm{ds}}(u_{1})\,\rd x_{0}^{\mathrm{ds}}(u_{2})}
{\bigl(x_{0}^{\mathrm{ds}}(u_{1})-x_{0}^{\mathrm{ds}}(u_{2})\bigr)^{2}}
=
\frac{\rd s\,\rd t}{2(s-t)^{2}}
\left(\frac{2+s+t}{2\sqrt{1+s}\sqrt{1+t}}-1\right).
\]
The scalar coefficient is one half of the expression in~\cite[Equation~(3.43)]{Mar08}, where the overall constant is left unspecified.

\newpage
 
 \begingroup
 \section*{List of notations}
 \small
 \renewcommand{\arraystretch}{1.08}
 \setlength{\LTleft}{0pt}
 \setlength{\LTright}{0pt}
 \begin{longtable}{@{}p{0.20\textwidth}@{\hspace{1em}}p{0.76\textwidth}@{}}
 	\textbf{Notation} & \textbf{Meaning} \\[0.4em]
 	\endfirsthead
 	\textbf{Notation} & \textbf{Meaning} \\[0.4em]
 	\endhead
 	$p$ 
 	& Integer parameter $p$ defining the family $X_p$. \\[0.3em]
 	$f$, $f_{\star}$
 	& Framing parameter and special value $f_\star=2/(p-1)$. \\[0.3em]
 	$\bbT'$; $\sfu$, $\sfv$
 	& Calabi--Yau subtorus and equivariant parameters. \\[0.3em]
 	$\mathbf{Q}$, $\mathbf{q}$
 	& A- and B-model closed string coordinates. \\[0.3em]
 	$X_i$, $\widehat{X}_i$
 	& A- and B-model open string coordinates. \\[0.3em]
 	$D$
 	& Logarithmic derivative $D=-\partial_t=\mathbf{Q}\partial_{\mathbf{Q}}$.\\[0.3em]
 	$\hbar$; $\mathbf{p}$
 	& Formal expansion parameter; power-sum variables $\mathbf{p}=(p_1,p_2,\ldots)$. \\[0.3em]
 	$\boldsymbol{\mu}$, $|\boldsymbol{\mu}|$
 	& Winding vector and total winding $|\boldsymbol{\mu}|=\sum_i\mu_i$. \\[0.3em]
 	$N_{g,d}$
 	& Genus-$g$, degree-$d$ closed GW invariants of $X_p$. \\[0.3em]
 	$N^{\mathrm{rel}}_{g,d,\mu}$, $N_{g,d,\mu}^{f}$
 	& Formal relative and open GW invariants with framing $f$. \\[0.3em]
 	$N_{0,0,(w)}^{f}$
 	& Degree-zero disk invariant $N_{0,0,(w)}^{f}=-\frac{1}{w^2}\binom{(f+1)w-1}{w-1}$, for $w\geq1$.\\[0.3em]
 	$\mathcal{F}_g$, $\check{\mathcal{F}}_g$
 	& A- and B-model closed potentials. \\[0.3em]
 	$\mathcal{F}_{g,n}^{f}$, $\check{\mathcal{F}}_{g,n}^{f}$
 	& A- and B-model open potentials. \\[0.3em]
 	$\mathcal{F}_{g,n}^{\bbT'}$
 	& $\bbT'$-equivariant twisted descendant series with disk factors. \\[0.3em]
 	$\mathcal{F}_{g,\boldsymbol{\mu}}^{f}$
 	& Coefficient of $X_1^{\mu_1}\cdots X_n^{\mu_n}$ in $\mathcal{F}^f_{g,n}$, after the closed mirror map. \\[0.3em]
 	$F^{f}$, $Z^{f}$
 	& Open-closed generating function and open-closed partition function. \\[0.3em]
 	$H$, $V_f$
 	& Equivariant lift of the hyperplane class; $V_f=\bbC(f)[H]/(H^2-f^2/4)$.\\[0.3em]
 	$C_f$
 	& Classical Yukawa coupling $C_f=\langle H^2,H\rangle_f$.\\[0.3em]
 	$L_{\alpha}^{f}$
 	& Eigenvalues of quantum multiplication by $H$.\\[0.3em]
 	$e_{\alpha}^{f}$, $\bar e_{\alpha}^{f}$
 	& Canonical idempotents and normalized canonical basis.\\[0.3em]
 	$\Delta_{\alpha}^{f}$
 	& $\Delta_{\alpha}^{f}=\langle e_{\alpha}^{f},e_{\alpha}^{f}\rangle_f^{-1}$.\\[0.3em]
 	$\Omega^{f}$, $\check{\Omega}^{f}$
 	& A- and B-model cohomological field theories. \\[0.3em]
 	$R^{f}(z)$, $\check{R}^{f}(z)$
 	& A- and B-model $R$-matrices. \\[0.3em]
 	$S^{f}(z)$, $\Phi_0^{f}$
 	& Specializations of the equivariant $S$-operator and the class $\Phi_0$.\\[0.3em]
 	$\mathcal{C}_f$; $x_f$, $y$
 	& Framed mirror curve $\mathcal{C}_f=(\bbP^1,x_f,y)$, with $x_f$ and $y$ defined in~\eqref{equ:SpecCurv}.\\[0.3em]
 	$\gamma_f$, $N_f(z)$
 	& $\gamma_f=f(p-1)-1$; $N_f(z)$ is the polynomial numerator of $x_f'(z)$ in~\eqref{equ:xf_prime_Nf}.\\[0.3em]
 	$a_{\alpha}$, $\alpha=0,1$
 	& Critical points of $x_f$ for generic $f$.\\[0.3em]
 	$K_f$
 	& $K_f=f/2+(p+1)\mathbf{q}/(1-p^2\mathbf{q})$, satisfying $L_{\alpha}^{f}=K_f-Dx_f(a_{\alpha})$.\\[0.3em]
 	$\varphi_i(z)$, $i=0,1$
 	& Rational functions defined by the partial fraction expansions in~\eqref{equ:varphi_i_partial_fraction}.\\[0.3em]
 	$\omega_{g,n}^{f}$
 	& Eynard--Orantin differentials of the mirror curve $\mathcal{C}_f$. \\[0.3em]
 	$\rd\zeta_k^{\bar\alpha}$, $\zeta_k^{\bar\alpha}$
 	& Meromorphic differentials of the second kind and their primitives.\\[0.3em]
 	$\rho(\widehat X)$
 	& Local inverse of $\widehat X(z)=\ee^{x_f(z)}$ at $z=0$.\\[0.3em]
 	$C_{w,j}^{f}$
 	& Coefficients in the expansion~\eqref{equ:B_disk_lagrange_coefficients} of the B-model disk potential.\\[0.3em]
 	$P_w^{f}(\mathbf{q})$
 	& $P_w^f(\mathbf q)=\sum_{j=0}^w(C_{w,j}^f/C_{w,0}^f)\mathbf q^j$.\\[0.3em]
 	$V_w^{f}$, $W_w^{\rB}$
 	& A- and B-model insertion vectors of winding $w$.\\[0.3em]
 	$\mathcal{U}_{\rA}^{f}$, $\mathcal{U}_{\rB}^{f}$
 	& A- and B-model insertion series.\\[0.3em]
 	$\mathcal{L}_{\rA}^{f}$, $\mathcal{L}_{\rB}^{f}$
 	& A- and B-model open leg contributions.\\[0.3em]
 	$\Delta$, $L$
 	& $\Delta=(1-p^{2}\mathbf{q})/(1-\mathbf{q})$ and $L=\Delta^{-1/2}=1+O(\mathbf{q})$. \\[0.3em]
 	$\calP_g$
 	& Polynomial numerator of the closed potential.\\[0.3em]
 	$\mathcal{P}_{g,\boldsymbol{\mu}}$
 	& Polynomial numerator of the normalized fixed-winding open potential.\\[0.3em]
 	$\mathbf{q}_{\varepsilon}$, $\iota_{\varepsilon}$
 	& $\mathbf{q}_{\varepsilon}=(1-\varepsilon^2\tau)/(p^2-\varepsilon^2\tau)$ and $\iota_{\varepsilon}(u)=-p+\varepsilon u$.\\[0.3em]
 	$F_g^{f,\tau,\mathrm{ds}}$
 	& Genus-$g$ free energy of the double-scaled curve with parameter $\tau$. \\
 \end{longtable}
 \endgroup

\vspace{2cm}

\newcommand{\etalchar}[1]{$^{#1}$}

\end{document}